\documentclass[letterpaper, bothsides, 12pt]{article} 
\usepackage[letterpaper,left=2.5cm,top=2.5cm,right=2.5cm,bottom=2.5cm]{geometry}  
\usepackage[utf8]{inputenc}
\usepackage{latexsym} 
\usepackage{layout}
\usepackage{float}
\usepackage{textcomp}
\usepackage{amstext}
\usepackage[pdftex]{graphicx}
\usepackage{multicol} 
\usepackage{longtable} 
\usepackage{lscape} 
\usepackage{multirow} 
\usepackage{array} 
\usepackage{amssymb}
\usepackage{amsmath} 
\usepackage{amsthm}
\usepackage[T1]{fontenc} 
\usepackage{listings} 
\usepackage[center, small]{caption2} 
\usepackage{graphicx}
\usepackage{verbatim}

\newtheorem{theorem}{Theorem}[section]
\newtheorem{lemma}[theorem]{Lemma}

\theoremstyle{definition}
\newtheorem{definition}[theorem]{Definition}

\theoremstyle{remark}
\newtheorem{remark}[theorem]{Remark}
\theoremstyle{notation}
\newtheorem{notation}[theorem]{Notation}

\theoremstyle{prop}
\newtheorem{prop}[theorem]{Propersition}

\theoremstyle{coexample}

\usepackage{titlesec}
\theoremstyle{definition}
\usepackage[pdftex]{color}
\usepackage[pdftex,colorlinks=true,linkcolor=blue,citecolor=black,urlcolor=cyan]{hyperref}

\newlength{\LL}

\title{ \Huge \bf Notes on generalized special Lagrangian equation}
\author{Xingchen Zhou\footnote{zxc3zxc4zxc5@stu.xjtu.edu.cn}
~~ \\ {\it Department of Mathematical Sciences} \\ {\it Tsinghua University}
}

\date{\today}

\begin{document}
\renewcommand{\tablename}{Tabla}

\maketitle

\hrulefill

\begin{abstract}
We obtain a prior $C^{1,1}$ estimates for some Hessian (quotient) equations with positive Lipschitz right hand sides, through studying a twisted special Lagrangian equation. The results imply the interior $C^{2,\alpha}$ regularity for $C^0$ viscosity solutions to $\sigma_2=f^2(x)$ in dimension 3, with positive Lipschitz $f(x)$.
\end{abstract}



\section{Introduction}
In this paper we will study the following generalized special Lagrangian equation
\begin{equation}\label{gslag1}
\sum_{i=1}^{n} \arctan \frac{\lambda_i}{f(x)}= \Theta
\end{equation}
in general dimensions $n\ge 3$, $\Theta$ is the phase constant. Let $u$ be its solution, we denote $\lambda_1\ge\lambda_2\ge\cdots\ge\lambda_n$ the eigenvalues of $D^2u$. We are interested if the twist $f$ is only positive and Lipschitz continuous. When $f=1$ it is the classical special Lagrangian equation that originates in the special Lagrangian geometry of Harvey-Lawson \cite{HL82} and enjoys good regularity.
Equation (\ref{gslag1}) is equivalent to the following Hessian (quotient) equations with $(n-1)-$convex solutions ($n=3,4$; $\Theta=\frac\pi2,\pi$ respectively)
\begin{equation}\label{gslag2}
\left\{
\begin{array}{lll}
\sigma_2=f^2(x)      &    \ n=3, & (a)\\
\frac{\sigma_3}{\sigma_1}=f^2(x)      & \   n=3,4. & (b)\\
\end{array} \right.
\end{equation}

Following Yuan \cite{Yua06}, we call the phase $|\Theta|= \frac\pi2(n-2)$ critical ("$>$" supercritical and "$<$" subcritical) since the level set $\{\lambda\in \mathbb{R}^n|\arctan\lambda=c\}$ is convex only when $|c|\ge \frac\pi2(n-2)$.  Singular solutions on subcritical phases exist, see Nadirashvili-Vl$\breve{a}$du$\c{t}$ \cite{NV10}, Wang-Yuan \cite{WY13}, and Mooney-Savin \cite{MS23} for non $C^1$ examples.\\

We state our main results in the following:

\begin{theorem}\label{Thm_Upbd_critical}
Let $u$ be a smooth solution to equation (\ref{gslag1}) in $B_{4}(0)$ with $|\Theta|\ge \frac\pi2 (n-2)$, $n\ge 3$. Suppose that $f$  is smooth and positive in $B_{4}(0)$, then we have
$$
|D^2u(0)|\le C(n,\mathrm{osc}_{B_{3}(0)} u, ||f^{-1}||_{L^{\infty}(B_{3}(0))}, ||f||_{C^{0,1}( B_{3}(0))}).
$$
\end{theorem}
\begin{theorem} \label{Thm2} Let u be a smooth solution to equation (\ref{gslag1}) in $B_{2}(0)$ with $|\Theta|\ge \frac\pi2 (n-2)$, $n\ge 3$. Suppose that $f$ is smooth and positive in $B_{2}(0)$, then we have
$$
\ |Du(0)| \leq C(n)(\|{Df}/{f}\|^2_{L^\infty(B_1(0))}+1)\mathop{osc}\limits_{B_1(0)} u.
$$
\end{theorem}
 A simple example $w=\frac1{\alpha(1+\alpha)}|x_1|^{1+\alpha}+\frac12(x_2^2+x_3^2)$ shows that Theorem \ref{Thm_Upbd_critical} is sharp for the quotient equation (\ref{gslag2}b). The situation for $\sigma_2$ equation remains open for us.
\begin{theorem}\label{coro_existence}
Let $f(x)\in C^{0,1}(\overline B_1(0))$, $f\ge1$ in $ B_{1}(0)$. Let $\varphi(x)\in C^0(\partial B_1(0)),\alpha\in (0,1)$. The following Dirichlet problem
$$ \left\{
\begin{array}{lcl}
F(D^2u,x)=\sum^n_{i=1}\arctan\frac{\lambda_i}{f}=\Theta      &      & B_1(0),\\
u=\varphi      &      & \partial B_1(0).
\end{array} \right. $$
has an unique solution $u\in C^0(\overline B_{1}(0))\cap C^{2,\alpha}(B_{1}(0))$ if the constant $\Theta\in[\frac\pi2(n-2),\frac\pi2n)$.
\end{theorem}
\noindent  One application of the above results is the interior regularity for $C^0$ viscosity solutions to equation (\ref{gslag2}a),
\begin{theorem}[Interior Regularity] \label{Thm_Regular_sigma2_3d}
Let $u$ be a $C^0$ viscosity solution for three dimensional equation $\sigma_2=f^2(x)$ in $ B_{1}(0)$. Suppose that $f\in C^{0,1}( B_{1}(0))$, $f>0$ in $ B_{1}(0)$. Then $u\in C^{2,\alpha}(B_1(0))$ for each $\alpha\in(0,1)$, and
$$
||u||_{C^{2,\alpha}( B_{\frac12}(0))}\le C(3, \alpha, \mathrm{osc}_{B_{\frac34}(0)} u,||f^{-1}||_{L^\infty( B_{\frac34}(0))}, ||f||_{C^{0,1}(\overline B_{\frac34}(0))}).
$$
\end{theorem}

In the past, various integral techniques combined with Sobolev inequalities were developed in gradient estimates, for minimal surface equations and quasilinear equations, in Bombieri-De Giorgi-Miranda [BDM69], Ladyzhenskaya-Ural'tseva \cite{LU70}, Simon \cite{Sim76}, and also Trudinger \cite{Tru72} \cite{Tru73}. These approaches were further explored for Hessian estimates to special Lagrangian equations in Warren-Yuan \cite{WY09}, Chen-Warren-Yuan \cite{CWY09} and Wang-Yuan \cite{WY11}.\\

The generalized special Lagrangian equation (\ref{gslag1}) was first brought up by Yu Yuan in the study of $\sigma_2$ equations with variable right hand sides. In an earlier joint work with Yingfeng Xu, we tried to get Hessian estimates with $f\in C^{1,1}$ in Wang-Yuan's way. In this study we push our goal further, using integral techniques to reduce the minimal regularity of $f$ to Lipschitz in Hessian estimates. Previously this is known to be $C^{1,1}$ for (\ref{gslag2}), and some times also requires strict convex solutions. The $C^{1,1}$ condition is imageable, since generally one need to differentiate the equation twice to see the behavior of $D^2u$, but then the second order derivatives of $f$ appear. A key observation in this work is that the trace of $D^2u$ admits a strong subharmonicity, say Jacobi inequality on the Lagrangian graph which contains a divergence structure $\Delta f$. This gives us an advantage in estimating the Hessian, since under integral sense we have a chance to bound the term containing $\Delta f$ with the Lipschitz norm of $f$.\\

To perform such idea, two basic tools in need will be Sobolev inequality and mean value inequality on the Lagrangian graph $(x,Du)$. Given the special form of (\ref{gslag1}), it is natural to consider it in the weighted Euclidean space $(\mathbb{R}^n \times {\mathbb{R}}^n,\overline g= f^2(x)dx^2+dy^2)$ (Qiu \cite{Qiu17} made this observation when deriving Hessian estimates for (\ref{gslag2}a) for $ C^{1,1}$ $f(x)$). Recall we have these two inequalities on submanifolds of $\mathbb{R}^n$ in Michael-Simon \cite{MS73}, and on general manifolds in Hoffman-Spruck \cite{HS74}. The later requires a prior information for sectional curvature and volume, which are not available in our situation. We develop these tools by introducing an intuitive vector field on the Lagrangian graph. It determines regions proportional to geodesic balls on the Lagrangian graph, and leads to a monotonicity formula only concerning with the weighted volume. The weighted volume is in fact the integral of nonnegative sysmetric functions $\sigma_k$ each having a divergence structure and can be calculated. Then we use the covering in Michael-Simon \cite{MS73} to get the desired results.\\

Once we get the $C^{1,1}$ estimates, the equations become uniformly elliptic.  We refer to a prior $C^{2,\alpha}$ estimates by Yu Yuan technical in \cite{Yua01}, for equation (\ref{gslag2}a) with $C^\alpha$ $f$. The existence results can be found in Caffarelli-Nirenberg-Spruck \cite{CNS85}, Theorem 1.3, the comparison principle for continuous viscosity solutions is in Trudinger \cite{Tru90}, Theorem 2.2. Then the interior $C^{2,\alpha}$ regularity can be obtained by a standard approximating process.\\

The $C^{2,\alpha}$ estimates for equation (\ref{gslag1}) follow from the $C^{1,1}$ estimates and Caffarelli \cite{Caf89}. For completeness of the story, we also establish a Dirichlet problem for (\ref{gslag1}), following Caffarelli-Nirenberg-Spruck \cite{CNS85} and a lecture note on special Lagrangian equations by Yu Yuan, 2016.\\

Heinz \cite{Hei59} derived a Hessian bound for Monge-Amp\`ere type equation including equation (\ref{gslag1}) in dimension 2. Pogorelov \cite{Pog78} and Chou-Wang \cite{CW01} established Hessian estimates to Monge-Amp\`ere equations and $\sigma_k$ equations for $k\ge 2$ with certain strict convexity constraints.
Urbas \cite{Urb00} \cite{Urb01} obtained Hessian estimates for $W^{2,p}$ weak solutions to $\sigma_k$ equations in
terms of certain Hessian integrals. Bao-Chen-Guan-Ji \cite{BCGJ03} obtained pointwise Hessian estimates to
strictly convex solutions for quotient equations $\sigma_n/\sigma_k$ in terms of certain integrals of the Hessian. Pointwise Hessian estimates for $\sigma_2=1$ in dimension 3 were obtained in Warren-Yuan \cite{WY09}. Qiu \cite{Qiu17} ($n=3$) also Guan-Qiu \cite{GQ19} ($u$ is convex or $\sigma_3>-A$) proved Hessian estimates for
quadratic Hessian equations with a $C^{1,1}$ variable right hand side. Hessian estimates for semi-convex solutions to $\sigma_2=1$ were
derived by a compactness argument in McGonagle-Song-Yuan \cite{MSY19}. Shankar-Yuan \cite{SY20a} \cite{SY20b} \cite{SY22} proved regularity for almost convex viscosity solutions, semi-convex solutions and semi-convex entire solutions to $\sigma_2=1$ respectively. Mooney \cite{Moo21} proved regularity for convex viscosity solutions to $\sigma_2=1$ using a new approach.
\begin{notation}
\indent The following notations will be used throughout this paper.
  \\ \indent  $u_i=\partial_iu,$ $u_{ij}=\partial_{ij}u,$ and $\lambda_{i,j}=\partial_j\lambda_i, \lambda_{i,jk}=\partial_{jk}\lambda_i$, etc.
   \\ \indent $V=\Pi^n_{i=1}\sqrt{1+\lambda_i^2}$ is the volume element.
   \\ \indent $T_{g^{-1}}=\sum^n_{i=1}\frac1{1+\lambda_i^2}$ is the trace of the inverse metric.
      \\ \indent $|f|_\infty$ is the $L^\infty$ norm of $f$ when the region is implicit.
   \\ \indent $C_f$ denotes various constants depending only on $n$ and $||f||_{C^{0,1}}$.
\end{notation}

\section{ preliminary inequalities on the Lagrangian graph}
\indent In this section we will study the geometric futures of the Lagrangian graph $(x,Du)$, and prove two basic inequalities on it: Sobolev inequality and mean value inequality. Here we consider the general dimension $n\ge 3$.
\subsection{Preliminaries}
\indent We introduce in the following the weighted Euclidean space $(\mathbb{R}^n \times {\mathbb{R}}^n,\overline g= f^2(x)dx^2+dy^2)$.
Let $z=(x,y)$, $x,y\in \mathbb{R}^n$, the Christoffel symbols of the weighted Euclidean space are
$$
\begin{aligned}
\overline\Gamma^k_{ij}&=\frac12{\overline g^{kl}}(\frac{\partial \overline g_{lj}}{\partial{z_i}}+\frac{\partial \overline g_{il}}{\partial{z_j}}-\frac{\partial \overline g_{ij}}{\partial{z_l}}) \\
   &= \frac1f(f_i\delta_{kj}+f_j\delta_{ik}-f_k\delta_{ij}).
\end{aligned}
$$
We have $\overline\Gamma^k_{ij}=0$ when one of the indexes $k,i,j>n$.
Equipped with a metric $g$ induced from the weighted Euclidean space, the graph $(\mathcal{M},g)$ becomes a submanifold which is not necessarily minimal, in contrast to the classical Lagrangian graph in Euclidean space. The tangent vector of $\mathcal{M}$ is
$$
\partial_i=\partial_i(x,Du)=\partial_{x_i} + \sum_\alpha D_\alpha u_i\partial_{y_\alpha},
$$
where $\partial_{x_i}$, $\partial_{y_\alpha}$ are unit tangent vectors $\mathbf{e}_i\in\mathbb{R}^{2n}$, $\mathbf{e}_{\alpha+n}\in\mathbb{R}^{2n}$ respectively.
 Choose a coordinate system such that $D^2u$ is diagonal at a point $p\in \mathcal{M}$, the induced metric is
$$ g_{ij}=f^2\delta_{ij}+Du_i\cdot Du_j\stackrel{p}{=}(f^2+\lambda_i^2)\delta_{ij}.$$
The the Beltrami-Laplace operator on $\mathcal{M}$ is given by
\begin{equation}\label{eq_Delta_g}
\begin{aligned}
\Delta_g &=\frac{1}{\sqrt g}\partial_j(\sqrt gg^{ij}\partial_i)
         =g^{ij}\partial_{ij}+\partial_jg^{ij}\partial_i+\frac{1}{2g}\partial_jg g^{ij}\partial_i \\
         &=g^{ij}\partial_{ij}-g^{il}g^{sj}\partial_j g_{ls} \partial_i + \frac{1}{2}g^{ls}g^{ij}\partial_j g_{ls}\partial_i\\
         &\stackrel{p}{=}g^{ii}\partial_{ii}+g^{ii}g^{jj}(ff_i-2ff_j\delta_{ij}-\lambda_i\lambda_{j,i})\partial_i\\
         &\stackrel{\Delta}{=}g^{ii}\partial_{ii}+\psi^ig^{ii}\partial_i,
\end{aligned}
\end{equation}
where $\psi^i\stackrel{p}{=}\sum_{j=1}^n g^{jj}(ff_i-2ff_j\delta_{ij}-\lambda_i\lambda_{j,i})$. Next we calculate the mean curvature of $\mathcal{M}$.
The Riemannian connection $\overline \nabla$ on the weighted Euclidean space $\mathbb{R}^n \times {\mathbb{R}}^n$ is
$$
\begin{aligned}
\overline \nabla_{\partial_i}{\partial_j} &=\overline \nabla_{(\partial_{x_i}+\sum_\alpha D_\alpha u_i\partial_{y_\alpha})}
   {(\partial_{x_j}+\sum_\beta D_\beta u_j\partial_{y_\beta})}
  {=}\overline \Gamma^k_{ij}\partial_{x_k}+\sum_\beta D_\beta u_{ij}\partial_{y_\beta}\\
  &{=}\sum_k \frac1f(f_i\delta_{kj}+f_j\delta_{ik}-f_k\delta_{ij})(e_k,0)+\sum_\beta D_\beta u_{ij}(0,e_\beta).\\
\end{aligned}
$$
Let $e_s\in \mathbb{R}^n$ be the unit vector, the unit normal vector in $(T\mathcal{M})^\perp$ is
$$N_s=\frac{(-Du_s,f^2e_s)}{f\sqrt{|Du_s|^2+f^2}}\stackrel{p}{=}\frac{(-\lambda_se_s,f^2e_s)}{f\sqrt{\lambda_s^2+f^2}}.$$
The mean curvature of $\mathcal{M}$ is
$$
\begin{aligned}
H&\stackrel{p}=g^{ii}<\overline \nabla_{\partial_i}{\partial_i},N_j>N_j
 \stackrel{p}{=}f^{-2}g^{ii}g^{jj}<\overline \nabla_{\partial_i}{\partial_i}, \ (-\lambda_je_j,f^2e_j)>(-\lambda_je_j,f^2e_j)\\
 &\stackrel{p}{=}\sum_k< \frac1f(2f_i\delta_{ki}-f_k)(e_k,0)+ \lambda_{i,k}(0,e_k),\ (-\lambda_je_j,f^2e_j)>
 f^{-2}g^{ii}g^{jj}(-\lambda_je_j,f^2e_j)\\
&\stackrel{p}{=}g^{ii}g^{jj}[\frac1f\lambda_j(-2f_i\delta_{ji}+f_j)+ \lambda_{i,j}](-\lambda_je_j,f^2e_j).
\end{aligned}
$$
\begin{lemma}\label{lemma_00}
The mean curvature of the gradient graph $\mathcal{M}$ is bounded by
\begin{align*}
|H|\le C(n)|Df|\sum^n_{i=1}\frac1{f^2+\lambda_i^2}=C(n)|Df|T_{g^{-1}}.
\end{align*}
\begin{proof}
Choose a coordinate system such that $D^2u$ is diagonal at a point $p\in \mathcal M$. Differentiate equation (\ref{gslag1}) once with respect to $x_j$  at $p$:
\begin{equation}\label{eq_uijj}
\sum_{i=1}^n \frac{\lambda_{i,j}f-\lambda_i f_j}{f^2+\lambda_i^2}=0,\ \ or\
\lambda_{i,j}g^{ii}=\frac1f\lambda_i f_jg^{ii}.
\end{equation}
By the formula of $H$ and (\ref{eq_uijj}), we have
$$
\begin{aligned}
|H|&\stackrel p{=}
g^{ii}g^{jj}\frac1f[\lambda_j(-2f_i\delta_{ji}+f_j)+ \lambda_if_j](-\lambda_je_j,f^2e_j)|
\le C(n)|Df|\sum^n_{i=1}\frac1{f^2+\lambda_i^2}.
\end{aligned}
$$
\end{proof}
\end{lemma}
Notice that if $Df\equiv0$, then $H\equiv0$.

\subsection{Sobolev inequality on $M$}
\ \\
\indent We introduce the following definition from Hoffman-Spruck \cite{HS74}:
\begin{definition}\label{def_01}
The divergence of a vector field X: $\mathcal{M}\rightarrow T(\mathbb{{R}}^{n}\times\mathbb{{R}}^{n})$ on $\mathcal{M}$ is the trace of $\overline \nabla X$ on $T\mathcal{M}$. In local coordinates, we have
$$
\mathrm{div}_\mathcal{M}X=\sum_{i,j=1}^ng^{ij}<\overline \nabla_{\partial_i}X,\ \partial_j>.
$$
\end{definition}
\begin{lemma}[Hoffman-Spruck \cite{HS74}]
If $X:M\rightarrow T(\mathbb{R}^n\times\mathbb{R}^n)$ is a vector field on $\mathcal{M}$, $f$ is a $C^1$ function on ${\mathcal{M}}$, then

$\mathrm{\romannumeral1)}$ $\int_{\mathcal{M}} \mathrm{div}_{\mathcal{M}}X dv_g=-\int_{\mathcal{M}}<X,H> dv_g+\int_{\mathcal{\partial M}}<X,\nu>dv_g$,\\
\indent $\mathrm{\romannumeral2)}$ $\mathrm{div}_{\mathcal{M}}fX=f\mathrm{div}_{\mathcal{M}}X+<X,\nabla_g f>$,\\
where $\nu$ is the exterior normal vector field to $\mathcal{M}$ on $\partial \mathcal{M}$.
\end{lemma}
Fix a point $z^0\in \mathcal{M}$, consider the vector field $z\in \mathcal{M}\rightarrow T(\mathbb{{R}}^{n}\times \mathbb{{R}}^{n})$: $Z(z,z_0)=\sum_{l=1}^{2n}(z_l-z_l^0)\partial_{z_l}$, and denote $r=|Z|$. Notice that it is not necessarily the usual radial vector field unless $f$ is a constant. We are interested in the divergence of $Z$. Choose a coordinate system such that $D^2u$ is diagonal at $z$, by definition \ref{def_01} we have
$$
\begin{aligned}
\mathrm{div}_\mathcal{M}\partial_{z_l}&=\sum^n_{i,j=1}g^{ij} <\overline \nabla_{\partial_i}\partial_{z_l},\ \partial_j>
=\sum^n_{i,j=1}\sum^{2n}_{k,s=1}g^{ij} <\overline \nabla_{\frac{\partial z_k}{\partial x_i}\partial_{z_k}}\partial_{z_l},\
\frac{\partial z_s}{\partial x_j}\partial_{z_s}>\\
&=\sum^n_{i,j=1}\sum^{2n}_{k,m,s=1}g^{ij}\overline g_{ms}\frac{\partial z_k}{\partial x_i}\frac{\partial z_s}{\partial x_j}\overline\Gamma^m_{kl}.
\end{aligned}
$$
Notice that $\overline \Gamma^m_{kl}=0$ when one of the indexes $m,k,l>n$, we have
\begin{equation}\label{eq_temp_21}
\begin{aligned}
\mathrm{div}_\mathcal{M}Z&=\sum^{2n}_{l=1}(z_l-z^0_l) \mathrm{div}_\mathcal{M}\partial_{z_l}+\sum^{2n}_{l=1}<\partial_{z_l}, \nabla_g (z_l-z_l^0)>\\
&=\sum^{2n}_{l=1}(z_l-z^0_l) \mathrm{div}_\mathcal{M}\partial_{z_l}+\sum^{2n}_{l,k=1}\sum^{n}_{i,j=1}<g^{ij} \frac{\partial{z_l}}{\partial{x_i}}\frac{\partial{z_k}}{\partial{x_j}}\partial{z_k},\partial_{z_l}>\\
&\stackrel z{=}n+f^2\sum^n_{i,l=1}(z_l-z^0_l)g^{ii}\overline\Gamma^i_{il}.
\end{aligned}
\end{equation}

In the following we will prove the main theorem in this section.

\begin{prop}[Sobolev inequality on $\mathcal{M}$]\label{prop_Sobolev}
Let $\phi$ be a non-negative $C^1(U)$ function vanishing outside a compact subset of $U\subset \mathcal{M}$. Suppose that $f\ge 1$. Then there exists a constant $C=C(n,||f||_{C^{0,1}(U)})$ such that
$$
[\int_{\mathcal{M}} \phi^{\frac n{n-1}}dv_g]^{\frac{n-1}n}\le C[1+(\int_{\mathrm{supp}\ \phi}T_{g^{-1}}dv_g)^{\frac1{n(n-1)}}]\int_{\mathcal{M}}[|\nabla_g\phi|+\phi(|H|+\sqrt{T_{g^{-1}}})]dv_g.
$$
\end{prop}
\begin{lemma} \label{lemma_01}
Suppose $\chi$ is a non-decreasing $C^1(\mathbb{R})$ function that $\chi(t)=0$ when $t\le 0$. Let $z_0\in \mathcal{M}$, $\rho>0$ be a constant, define functions $\varphi_0,\ \psi_0,$ by
$$
\begin{aligned}
&\varphi_0(\rho)=\int_\mathcal{M}\phi(x)\chi(\rho-r)dv_g,\\
&\psi_0(\rho)=\int_\mathcal{M}[|\nabla_g\phi(x)|+\phi(x)(|H|+\sqrt{T_{g^{-1}}})]\chi(\rho-r)dv_g,
\end{aligned}
$$
where $r=|Z|$ is the norm of the vector with respect to $z_0$. Then there exists a constant $C=C(n)|Df(1+f+|Df|)^3|_{\infty}$ such that for all $\rho>0$,
$$
-\frac{d}{d\rho}(\frac{\varphi_0(\rho)}{\rho^n})\le (1+C)\frac{\psi_0(\rho)}{\rho^n}+C\frac1{\rho^{n-1}}\int_\mathcal{M}\sqrt{T_{g^{-1}}} \phi(x)\chi'(\rho-r)dv_g.
$$
\end{lemma}
\begin{proof}
We begin by calculating that
$$
\begin{aligned}
&\mathrm{div}_\mathcal{M}[\chi(\rho-r)\phi Z]
=\chi(\rho-r)\phi \mathrm{div}_\mathcal{M}Z+<\nabla_g[\chi(\rho-r)\phi],Z>\\
&=\chi(\rho-r)\phi \mathrm{div}_\mathcal{M}Z-\chi'(\rho-r)\phi<\nabla_gr,Z>+\chi(\rho-r)<\nabla_g\phi,Z>.
\end{aligned}
$$
We integrate the above equality over $\mathcal{M}$ and get
\begin{equation}\label{eq_temp_22}
\begin{aligned}
&\int_\mathcal{M}\chi(\rho-r)\phi\mathrm{div}_\mathcal{M} Zdv_g
=-\int_\mathcal{M}\chi(\rho-r)\phi<Z,H>dv_g\\
&+\int_\mathcal{M}\chi'(\rho-r)\phi<\nabla_gr,Z>dv_g-\int_\mathcal{M}\chi(\rho-r)<\nabla_g\phi,Z>dv_g.
\end{aligned}
\end{equation}
Define functions $\alpha_k(t)= t$ when $k\le n $, $\alpha_k(t)=1$ when $k>n$, and $\beta_k(t)= t$ when $k\le n $, $\beta_k(t)=0$ when $k>n$. Choose a coordinate system such that $D^2u$ is diagonal at $z$. We have $r=\sqrt{\sum^{2n}_{k=1}\alpha_k(f^2)(z_k-z^0_k)^2}$, notice $f\ge1$ implies $T_{g^{-1}}\le n\sqrt{T_{g^{-1}}}$,
\begin{equation}\label{eq_temp_23}
\begin{aligned}
|\nabla_gr|^2&\stackrel z{=}\sum^n_{i=1}\frac{g^{ii}}{r^2}
[\sum^{2n}_{k=1}\alpha_k(f^2)\frac{\partial{z_k}}{\partial{x_i}}(z_k-z_k^0)+\sum^{2n}_{k=1}\beta_{k}(ff_i)(z_k-z_k^0)^2]^2\\
&\stackrel z{\le}1+C(n)[\rho(|fDf|+|f^3Df|)\sqrt{T_{g^{-1}}}+\rho^2|fDf|^2 T_{g^{-1}}) \le (1+C_1\rho\sqrt{T_{g^{-1}}})^2,
\end{aligned}
\end{equation}
where the coefficient $C_1=C(n)|Df(1+f+|Df|)^3|_{\infty}$.
Now we use (\ref{eq_temp_21}), (\ref{eq_temp_23}) to estimate equality (\ref{eq_temp_22}). Notice that $|\overline \Gamma_{ij}^i|\le \frac{|Df|}{f}$, by (\ref{eq_temp_21}) we have $|\mathrm{div}_\mathcal{M} Z|\ge n-C(n)|fDf|r T_{g^{-1}}$,
$$
\begin{aligned}
&\int_\mathcal{M}\chi(\rho-r)\phi(x)[n-C(n)\rho|fDf| T_{g^{-1}}]dv_g
\le\int_\mathcal{M}\chi(\rho-r)\phi(x) \rho |H|dv_g\\
&+\int_\mathcal{M}\chi'(\rho-r)\phi(x) (1+C_1\rho\sqrt{T_{g^{-1}}})\rho dv_g+\int_\mathcal{M}\chi(\rho-r)\rho|\nabla_g\phi(x)|dv_g.
\end{aligned}
$$
Thus we have
$$
n\varphi_0(\rho)-\rho\varphi_0'(\rho)\le (1+C_1)\rho\psi_0(\rho)+C_1\rho^2\int_\mathcal{M}\sqrt{T_{g^{-1}}} \phi(x)\chi'(\rho-r)dv_g.
$$
Multiply the above inequality by $\rho^{-n-1}$ and we get the conclusion.
\end{proof}

\begin{remark}
If $Df\equiv 0$, we have $\mathrm{div}_\mathcal{M}Z =n$, $C_1=0$, then Lemma \ref{lemma_01} reduces to Michael-Simon \cite{MS73}, Lemma 2.2. We can show that this is still true in the following proof by accurately calculating the related coefficients, but for simplicity we won't do this. We only point out that to get a consistent result to Michael-Simon \cite{MS73}, the critical step is the monotonicity formula,
$$
-\frac{d}{d\rho}(\frac{\varphi_0(\rho)}{\rho^n})\le \frac{\psi_0(\rho)}{\rho^n}.
$$
\end{remark}
In the following we will denote $S_\rho(z_0)=\{z\in \mathcal{M}|\ |Z(z,z_0)|\le \rho\}$, where the vector $Z$ is defined as before. Notice that since $f\ge1$, we have $ S_\rho(z_0)\supset \mathcal{B}_\frac{\rho}{|f|_\infty}(z_0)$, the later is the geodesic ball on $M$ with radius $\frac{\rho}{|f|_\infty}$.
\begin{lemma}\label{lemma_02}
Let $U\subset \mathcal{M}$ be an open subset. Let $\phi(x)\in C_0^1(U)$ satisfy $0\le\phi\le 1$, $z_0\in U$ be a point such that $\phi(z_0)=1$, $\lim_{\rho\rightarrow 0}\rho^{-n}|S_\rho(z_0)|\ge(||f||_{L^\infty(U)})^{-n} \omega_n$. Let $\tau>1$ be an arbitrary constant. Define $\overline\varphi_0,\ \overline\psi_0,\ J_\phi,\ T_\phi$ by
$$
\begin{aligned}
&\overline \varphi_0(\rho)=\int_{S_\rho(z_0)}\phi(x)dv_g,\\
&\overline\psi_0(\rho)=\int_{S_\rho(z_0)}[|\nabla_g\phi(x)|+\phi(x)(|H|+\sqrt{T_{g^{-1}}})]dv_g,\\
&J_\phi=[\omega_n^{-1}\int_{\mathcal{M}}\phi(x) dv_g]^{\frac 1n},\ T_\phi=\int_{\mathrm{supp} (\phi)}T_{g^{-1}} dv_g.
\end{aligned}
$$
then there exist positive constants $\mu$, $\kappa$ depending only on $||f||_{C^{0,1}(U)} $ such that for some constant $0<\rho<\kappa(1+ T_\phi^{\frac1{n(n-1)}})J_\phi $,
$$
\overline \varphi_0(\tau\rho)\le \mu\tau^{n}(1+ T_\phi^{\frac1{n(n-1)}})J_\phi \overline \psi_0(\rho).
$$
\end{lemma}
\begin{proof}
Let $\rho_0=\kappa J$ where $\kappa$ is to be chosen later, let $\sigma\in(0,\rho_0)$ be a constant. By Lemma \ref{lemma_01} we have
$$
\begin{aligned}
\sigma^{-n}\varphi_0(\sigma)&\le\rho_0^{-n}\varphi_0(\rho_0)+C[\int^{\rho_0}_0 \rho^{-n}\psi_0(\rho)d\rho+ \int^{\rho_0}_{\sigma}\frac1{\rho^{n-1}}\int_\mathcal{M}\sqrt{T_{g^{-1}}} \phi(x)\chi'(\rho-r)dv_gd\rho]\\
&\le\rho_0^{-n}\varphi_0(\rho_0)+C_2[\int^{\rho_0}_0 \rho^{-n}\psi_0(\rho)d\rho+\rho_0^{1-n}\int_\mathcal{M}\sqrt{T_{g^{-1}}} \phi(x)\chi(\rho_0-r)dv_g],\\
\end{aligned}
$$
where $C=1+C(n)|Df(1+f+|Df|)^3|_{\infty}$ and $C_2=C(n)C$.
Let $\epsilon\in(0,\sigma)$, we require the function $\chi(t)$ in Lemma \ref{lemma_01} to satisfy $\chi(t)=1$ when $t\ge \epsilon$. It follows that $\overline\varphi_0(\sigma-\epsilon)\le\varphi_0(\sigma)\le\overline\varphi_0(\sigma)$, $\psi_0(\rho)\le\overline\psi_0(\rho)$, thus
$$
\sigma^{-n}\overline \varphi_0(\sigma-\epsilon)\le\rho_0^{-n}\overline\varphi_0(\rho_0)+C_2[\int^{\rho_0}_0 \rho^{-n}\overline\psi_0(\rho)d\rho+\rho_0^{1-n}\int_{S_{\rho_0}(z_0)}\sqrt{T_{g^{-1}}} \phi(x)dv_g].
$$
Since $\epsilon$, $\sigma$ are arbitrary, we have
\begin{equation} \label{eq_temp_24}
\sup_{\sigma\in(0,\rho_0)} \sigma^{-n}\overline \varphi_0(\sigma)\le\rho_0^{-n}\overline\varphi_0(\rho_0)+C_2[\int^{\rho_0}_0 \rho^{-n}\overline\psi_0(\rho)d\rho+\rho_0^{1-n}\int_{S_{\rho_0}(z_0)}\sqrt{T_{g^{-1}}} \phi(x)dv_g].
\end{equation}

To prove that inequality (\ref{eq_temp_24}) implies the conclusion, we argue by contradiction. Suppose for all $\rho\in(0,\rho_0)$, $\overline\psi_0(\rho)<\mu^{-1} \rho_0^{-1}\tau^{-n}\overline\varphi_0(\tau\rho)$. We estimate separately the terms at the right hand side of (\ref{eq_temp_24}).
We have
$$
\begin{aligned}
\int^{\rho_0}_0 \rho^{-n}\overline\psi_0(\rho)d\rho &\le \mu^{-1}\rho_0^{-1}\tau^{-n}\int_0^{\rho_0} \rho^{-n}\overline\varphi_0(\tau\rho)d\rho\\
&= \mu^{-1}\rho_0^{-1}\tau^{-1}\int_0^{\rho_0} \sigma^{-n}\overline\varphi_0(\sigma)d\sigma+\mu^{-1}\rho_0^{-1}\tau^{-1}\int_{\rho_0}^{\tau\rho_0} \sigma^{-n}\overline\varphi_0(\sigma)d\sigma\\
&\le \mu^{-1}\tau^{-1} \sup_{\sigma\in (0,\rho_0)}{\sigma^{-n}\overline\varphi_0(\sigma)}+\mu^{-1}\rho_0^{-n}\overline\varphi_0(\tau\rho_0).
\end{aligned}
$$
Recall $\phi\le 1$, $f\ge1$, we have
$$
\begin{aligned}
\rho_0^{1-n}\int_{S_\rho(z_0)}\sqrt{T_{g^{-1}}} \phi(x)dv_g &\le\rho_0^{1-n}(\int_{\mathrm{supp(\phi)}}{T_{g^{-1}}}^{\frac n2}dv_g)^{\frac 1n}(\int_{\mathrm{supp}(\phi)}\phi^{\frac{n}{n-1}}(x)dv_g)^{\frac {n-1}{n}}\\
&\le C(n)\omega_n^{\frac{n-1}{n}}\kappa^{1-n} T_\phi^{\frac 1n}.
\end{aligned}
$$
Let $C_3=C(n)\omega^{-\frac1n}C_2$, (\ref{eq_temp_24}) becomes
$$
(1-C_2\mu^{-1}\tau^{-1} ) \sup_{\sigma\in (0,\rho_0)}{\sigma^{-n}\overline\varphi_0(\sigma)}\le (1+C_2\mu^{-1} )\kappa^{-n}\omega_n+C_3\kappa^{1-n}T_\phi^{\frac 1n}\omega_n.
$$
Since $\tau>1$ and $\sup_{\sigma\in (0,\rho_0)}{\sigma^{-n}\overline\varphi_0(\sigma)}\ge \lim_{\rho\rightarrow 0}\rho^{-n}|S_\rho(z_0)|\ge (|f|_{\infty})^{-n}\omega_n$, we can choose
$$
\begin{aligned}
&\mu\ge4(1+C_2),\  \mu= C(n)[1+|Df(1+f+|Df|)^3|_{\infty}],\\
&\kappa^n \ge 8|f|^n_\infty,\ \kappa^{n-1}\ge 8C_3T_\phi^{\frac1n}|f|^n_\infty, \\
&\kappa=C(n)[1+|Df(1+f+|Df|)^3|_{\infty}]^{\frac1{n-1}}|f|_\infty^{\frac n{n-1}}(1+T_\phi^{\frac1{n(n-1)}}),
\end{aligned}
$$
and get a contradiction.

%
\end{proof}

Proof of Proposition 2.2:
\begin{proof}
We follow essentially the argument of Michael-Simon\cite{MS73}. Let $\alpha\ge\epsilon>0$ be constants, let $\hat\chi(t)$ be a non-decreasing $C^1(\mathbb{R})$ function such that $\hat\chi(t)=0$ when $t\le -\epsilon$, $\hat\chi(t)=1$ when $t\ge 0$. Consider the set $\mathcal{M}_\alpha=\{z\in \mathcal{M}|\ \hat\chi(\phi-\alpha)\ge 1\}\backslash\mathcal{M}_0$ where $\mathcal{M}_0$ is a set of measure zero that the condition $\lim_{\rho\rightarrow 0}\rho^{-n}|\mathcal{B}_\rho(z_0)|\ge \omega_n$ fails. We substitute $\phi(x)$ in Lemma \ref{lemma_02} by $\hat\chi(\phi-\alpha)$, let $\rho_i= 2^{-i+1}\kappa (1+ T_{\hat\chi(\phi-\alpha)}^{\frac1{n(n-1)}})J_{\hat\chi(\phi-\alpha)}$, $i=1,\ 2,...$, where we use the same definition as in Lemma \ref{lemma_02}. Notice that $\mathrm{supp}(\hat\chi(\phi-\alpha))\subset\mathrm{supp}(\phi)$, we denote
 $$
 \mathcal{M}_\alpha^i=\{z_0\in \mathcal{M_\alpha}|\ \overline \varphi_0(2\tau\rho)\le \mu(2\tau)^n (1+ T_\phi^{\frac1{n(n-1)}})J_{\hat\chi(\phi-\alpha)}\overline\psi_0(\rho)\ \mathrm{for}\ \mathrm{some}\ \rho\in( 2^{-1}\rho_i,\rho_i] \}.
$$
It follows that $\mathcal{M_\alpha}=\bigcup_i \mathcal{M}_\alpha^i$.
Next we define inductively a sequence of subsets ${F}_0,\ {F}_1...$ as follows:\\
\indent(\textit{1}) ${F}_0=\emptyset$;\\
\indent(\textit{2}) Let $k\ge 1$ such that ${F}_0,\ {F}_1,...,\ {F}_{k-1}$ have been defined. Let $G_k= \mathcal{M}^k_\alpha - \bigcup^{k-1}_{i=0} \bigcup_{z\in {F}_i} S_{\tau \rho_i}$. If $G_k=\emptyset$, then ${F}_k=\emptyset$. Else choose ${F}_k$ to be a finite subset of $G_k$ such that $G_k\subset \bigcup_{z\in {F}_k}S_{\tau \rho_k}(z)$ and $S_{\rho_k}(z)$, $z\in F_k$ are pairwise disjoint. We can choose $\tau=2|f|^2_\infty$, since $S_{2|f|_\infty\rho}(z_0)\supset \mathcal{B}_{2\rho}(z_0)\supset \mathcal{B}_{\rho}(z_0)\supset S_{\frac{\rho}{|f|_\infty}}(z_0)$ for all $\rho>0, z_0\in \mathcal{M}$, where $\mathcal{B}_\rho$ is geodesic the ball on $\mathcal{M}$. It follows that\\
\indent(a) $F_i\subset M_\alpha^i$, $i=1,2,...$,\\
\indent(b) $M_\alpha\subset \bigcup_{i=1}\bigcup_{z\in F_i}S_{\tau \rho_i}(z)$,\\
\indent(c) all the sets of $S_{\rho_i}(z)$, $z\in F_i$, $i=1,2,...$ are pairwise disjoint.\\
By (a) we have for each $z_0\in F_i$, $\overline\varphi_0(2\tau \rho)\le \mu (2\tau)^n(1+ T_\phi^{\frac1{n(n-1)}})J_{\hat\chi(\phi-\alpha)} \overline\psi_0(\rho)$ for some $\rho\in(2^{-1}\rho_i,\rho_i]$.
Summing over all $z\in F_i$ and $i=1,2,...$, and denote $ C_\phi= C(n)\mu \tau^n (1+ T_\phi^{\frac1{n(n-1)}})$, by (b) (c) we have
$$
\begin{aligned}
|M_\alpha|&\le C_{\phi} J_{\hat\chi(\phi-\alpha)}\int_M[|\nabla_g\hat\chi(\phi-\alpha)|+\hat\chi(\phi-\alpha)(|H|+\sqrt{T_{g^{-1}}})]dv_g\\
&\le C_\phi[\int_M\hat\chi(\phi-\alpha)dv_g]^{\frac 1n}\int_M[|\nabla_g\hat\chi(\phi-\alpha)|+\hat\chi(\phi-\alpha)(|H|+\sqrt{T_{g^{-1}}})]dv_g.
\end{aligned}
$$
We multiply the above inequality by $\alpha^{\frac 1{n-1}}$ and integrate over $(0,\infty)$,
$$
\int_M \phi^{\frac{n}{n-1}}dv_g\le
C_\phi[\int_M(\phi+\epsilon)^{\frac{n}{n-1}}dv_g]^{\frac{1}{n}} \int_M
\int_0^\infty[|\nabla_g\hat\chi(\phi-\alpha)|+\hat\chi(\phi-\alpha)(|H|+\sqrt{T_{g^{-1}}})] d\alpha dv_g,
$$
where we use the fact that $\hat\chi(\phi-\alpha)=0$ when $\phi+\epsilon\le \alpha$. Finally we calculate that
$$
|\nabla_g\hat\chi(\phi-\alpha)|=\hat\chi'(\phi-\alpha)|\nabla_g\phi|,
$$
$$
\begin{aligned}
\int_0^\infty\hat\chi(\phi-\alpha)d\alpha = \int_0^\infty\alpha\hat\chi'(\phi-\alpha)d\alpha
\le (\phi+\epsilon)\int_0^\infty \hat\chi'(\phi-\alpha)d\alpha
\le \phi+\epsilon,
\end{aligned}
$$
and let $\epsilon\rightarrow 0$ we get the desired result.
\end{proof}

\subsection{Mean value inequality on $\mathcal{M}$}
\ \\
\indent Let $U$ be an open subset of $M$. Consider a nonnegative integrable function $\phi\ge0$ on $M$ satisfying $\Delta_g \phi\ge0$ in the sense that for each non-negative function $\gamma\in C^\infty_0(U)$, the following inequality holds:
 \begin{equation}\label{eq_temp_25}
 \int_M \phi\Delta_g\gamma dv_g\ge 0.
\end{equation}
 The proof of the monotonicity formula in Lemma \ref{lemma_01} implies a mean value inequality for $\phi$ on $M$, which we will discuss in the following.
\begin{prop}\label{prop_Meanvalue}
Suppose that $\phi$ is a nonnegative integrable function on $M$ satisfying inequality (\ref{eq_temp_25}). There exists a constant C depending $||f||_{C^{0,1}(U)}$ such that for a.e. $z_0\in U$ and all $0<\rho\le 1$ with geodesic ball $\mathcal{B}_\rho(z_0)\subset\subset U$,
$$
\phi(z_0)\le C\rho^{-n}\int_{\mathcal{B}_\rho(z_0)}\phi(z)dv_g.
$$
\end{prop}
\begin{proof}
We may assume that $z_0=0$. Denote $r=|Z(z,0)|$. Let $\epsilon>0$ be an arbitrary constant, let $\chi(t)$ be a non-decreasing $C^1(\mathbb{R})$ function that $\chi(t)=0$ when $t\le 0$, and $\chi(t)=1$ when $t\ge \epsilon$. Define functions $\varphi_0$, $\overline\varphi_0$ as in Lemma \ref{lemma_01} and Lemma \ref{lemma_02}. Define a test function $\gamma$ by
$$ \gamma(s)=\int_s^\infty t\chi(\rho-t)dt.
$$
Define function $\alpha_k(t),\ \beta_k(t)$ as in the proof of Lemma \ref{lemma_01}. We calculate that
$$
\frac{\partial\gamma(r)}{\partial x_i}= -\chi(\rho-r)\sum^{2n}_{k=1}[\alpha_k(h^2)z_k\frac{\partial z_k}{\partial x_i}+\beta_k(hh_i)z_k^2],
$$
$$
\begin{aligned}
\frac{\partial^2\gamma(r)}{\partial x_i \partial x_j}&= -\chi(\rho-r)\sum^{2n}_{k=1}\alpha_k(h^2)[\frac{\partial z_k}{\partial x_i}\frac{\partial z_k}{\partial x_j}+z_k\frac{\partial^2 z_k}{\partial x_i \partial x_j}]\\
&-\chi(\rho-r)\sum^{2n}_{k=1}[2\beta_k(ff_j)z_k\frac{\partial z_k}{\partial x_i}+2\beta_k(ff_i)z_k\frac{\partial z_k}{\partial x_j}+\beta_k(ff_{ij}+f_if_j)z_k^2]\\
&+\frac1 r\chi'(\rho-r)\sum^{2n}_{k,s=1}[\alpha_k(f^2)z_k\frac{\partial z_k}{\partial x_i}+\beta_k(ff_i)z_k^2][\alpha_s(f^2)z_s\frac{\partial z_s}{\partial x_i}+\beta_s(ff_i)z_s^2].
\end{aligned}
$$
Choose a coordinate system such that $D^2u$ is diagonal at $z$, by the definition of $\Delta_g$ (\ref{eq_Delta_g}) and equality (\ref{eq_uijj}), we have
\begin{equation}\label{eq_temp_26}
\Delta_g \gamma(r)\le [-n+C(r+r^2)-f\Delta_gf|x|^2]\chi(\rho-r)+[r+C_f(r^2+r^3)]\chi'(\rho-r),
\end{equation}
where $C_f$ is a constant depending on $||f||_{C^{{0,1}}(U)}$.
We choose the test function in (\ref{eq_temp_25}) to be $\gamma(r)$,
$$
\int_\mathcal{M}\phi(x)\Delta_g \gamma(r)dv_g\ge 0 .
$$
Use inequality (\ref{eq_temp_26}) and the fact that $|\nabla_g r|\le C_f$, we have
$$
\begin{aligned}
&n\int_\mathcal{M}\phi(x)\chi(\rho-r)dv_g-\rho\int_\mathcal{M}\phi(x)\chi'(\rho-r)dv_g\\
 &\le C_f(\rho+\rho^2)\int_\mathcal{M}\phi(x)\chi(\rho-r)dv_g +C_f(\rho^2+\rho^3)\int_\mathcal{M}\phi(x)\chi'(\rho-r)dv_g.
\end{aligned}
$$
Notice that $\gamma(r)\le \rho^2\chi(\rho-r)$, $\rho\le 1$, we have
$$
-\frac{d}{d\rho}(\frac{\varphi_0(\rho)}{\rho^n})\le C[\frac{\varphi_0(\rho)}{\rho^n}+\frac{\varphi_0'(\rho)}{\rho^{n-1}}]=C_f[(n+1)\frac{\varphi_0(\rho)}{\rho^n}+\rho\frac{d}{d\rho}(\frac{\varphi_0(\rho)}{\rho^n})].
$$
Multiply the above inequality by $\rho^n[\varphi_0(\rho)]^{-1}$,
$$
-(1+C_f\rho)\frac{d}{d\rho}\ln (\frac{\varphi_0(\rho)}{\rho^n})\le C_f(1+n).
$$
Choose $\rho_0\in(0,1)$, and integrate the above inequality over $(\sigma,\rho_0)$, we get for arbitrarily small constants $0<\epsilon<\sigma$,
$$
\sigma^{-n} \overline \varphi_0(\sigma-\epsilon) \le e^{(n+1)C_f}\rho_0^{-n}\overline\varphi_0(\rho_0).
$$
Recall $\mathcal{B}_\frac{\rho_0}{|f|_\infty}(0)\subset S_{\rho_0}(0)\subset \mathcal{B}_{\rho_0|f|_\infty}(0)$, for general $0<\rho<1$ we have
$$
\phi(0)\le  C_f\rho^{-n}\int_{\mathcal{B}_\rho(0)}\phi(z)dv_g.
$$

\end{proof}

\section{Trace Jacobi Inequality For Gslag}

In this section we prove trace Jacobi inequalities on critical and supercritical phases for generalized special Lagrangian equation (\ref{gslag1}). It is another crucial step to get Hessian estimates with Lipschitz twist besides the Sobolev inequality. Notice that the result applies to convex solutions as well.

\begin{lemma}\label{lemma_Trace_Jacobi}
Let $u$ be a smooth solution to equation (\ref{gslag1}) with $\Theta\ge \frac\pi2(n-2)$. Suppose that $f$ is smooth, $f\ge 1$, denote $\Delta\cdot=\Delta u$. Then we have
\begin{equation}\label{eq_Trace_Jacobin1}
   \begin{aligned}
&\Delta_g\ln|A+\Delta \cdot|-\frac\epsilon2|\nabla_g\ln|A+\Delta \cdot||^2\\ &\ge3\delta|A+\Delta \cdot|^{-2}\sum^n_{\gamma,i=1}g^{ii}u^2_{ii\gamma}+ |A+\Delta \cdot|^{-1}\frac{\Delta f} f\sum^n_{i=1}g^{ii}\lambda_{i}-{C_f}T_{g^{-1}},
   \end{aligned}
\end{equation}
where $\epsilon, \delta, A,C_f$ are positive constants depending only on $n$ and $||f||_{C^{0,1}}$. The same conclusion holds when $u$ is convex.
\end{lemma}

%

By approximation we suppose that the eigenvalues of $D^2u$ are pairwise distinct at $p\in \mathcal{M}$. Choose a coordinate system such that $D^2u$ is diagonal at $p$,
\begin{equation}\label{eq_32}
\begin{aligned}
\lambda_{\gamma,ii}&=\sum^n_{s,t=1}\frac{\partial\lambda_\gamma}{\partial u_{st}}
                     \frac{\partial^2u_{st}}{\partial x^2_i}+\sum^n_{s,t,k,l=1}\frac{\partial^2\lambda_\gamma}{\partial u_{st}\partial u_{kl}}\frac{\partial u_{st}}{\partial x_i}
                     \frac{\partial u_{kl}}{\partial x_i}\\
   &=\lambda_{i,\gamma\gamma}+2(\sum_{t\neq \gamma}\frac1{\lambda_\gamma-\lambda_t}u^2_{it\gamma}-\sum_{t\neq i}\frac1{\lambda_i-\lambda_t}u^2_{it\gamma}).
\end{aligned}
\end{equation}
Taking derivative of equation (\ref{gslag1}) with respect to $x_\gamma$ twice, we have at $p$

\begin{equation}\label{eq_33}
\sum^n_{i=1}\lambda_{i,\gamma\gamma}g^{ii}=\sum^n_{i=1}\frac1f[2(ff_\gamma+\lambda_iu_{ii\gamma})(u_{ii\gamma}f-\lambda_if_\gamma)g^{ii}g^{ii}
+\lambda_if_{\gamma\gamma}g^{ii}].
\end{equation}
 Summing up (\ref{eq_32})$\times g^{ii} $ and (\ref{eq_33}), and denote $\Delta\cdot=\Delta u$, $\overline \Delta_g=g^{ij}\partial_{ij}\stackrel{p}{=}g^{ii}\partial_{ii}$,
\begin{equation}\label{eq_34}
\overline\Delta_g(\Delta\cdot)
\ge\sum^n_{\gamma=1}[\sum^n_{i=1}2\lambda_ig^{ii}g^{ii}u^2_{ii\gamma}+\sum^n_{i\neq \gamma}2(\lambda_i+\lambda_\gamma)g^{ii}g^{\gamma\gamma}u^2_{ii\gamma}]+L,
\end{equation}
where $L$ is the linear part of $u_{\clubsuit\clubsuit\spadesuit}$. Let $C_{f}=C(n,||f||_{C^{0,1}})$, $\delta=\delta(n)>0$, $A>1$,
\begin{align*}
    L&=-C_{f}\sum^n_{\gamma,i=1}g^{ii}(|u_{ii\gamma}|+1)+\frac{\Delta f} f\sum^n_{i=1}g^{ii}\lambda_{i}\\
&\ge -\delta|A+\Delta\cdot|^{-1} g^{ii}u^2_{ii\gamma}+\frac{\Delta f} f\sum^n_{i=1}g^{ii}\lambda_{i}-\frac {C_f}\delta T_{g^{-1}}|A+\Delta\cdot|.
\end{align*}
Recall $\Delta_g\stackrel{p}=g^{ii}\partial_{ii}+\psi^ig^{ii}\partial_i$, $|\psi^i|=|\sum^n_{j=1}g^{jj}(ff_i-ff_j\delta_{ij}-\lambda_i\lambda_{j,i})|
\le C_f\sum^n_{j=1}(1+|\lambda_i\lambda_j|)g^{jj}$. Let $\epsilon=\epsilon(n)>0$ be small, we have
$$
|(\Delta_g-\overline \Delta_g) \Delta\cdot|=|\sum^n_{i=1}\psi^ig^{ii}D_i\Delta\cdot|\le \frac\epsilon2 \frac{|\nabla_g\Delta\cdot|^2}{A+\Delta\cdot}+\frac {C_f}\epsilon T_{g^{-1}}|A+\Delta\cdot|,
$$
\begin{align*}
\Delta_g(\Delta\cdot)\ge \overline\Delta_g (\Delta\cdot)-\frac\epsilon2\frac{|\nabla_g\Delta\cdot|^2}{A+\Delta\cdot} -\frac{C_f}{\epsilon} T_{g^{-1}}|A+\Delta\cdot|.
\end{align*}
Now we prove a weak trace Jacobi inequality. We have at $p$
\begin{equation}\label{eq_35}
\begin{aligned}
&\Delta_g|A+\Delta\cdot|-(1+\frac\epsilon2) \frac{|\nabla_g\Delta\cdot|^2}{A+\Delta\cdot}\ge-\delta\sum^n_{\gamma,i=1}\frac{ g^{ii}u^2_{ii\gamma}}{A+\Delta\cdot}+ \frac{\Delta f} f\sum^n_{i=1}g^{ii}\lambda_{i}-\frac{C_f}{\epsilon\delta}T_{g^{-1}}{|A+\Delta\cdot|}\\
&+ \sum^n_{\gamma=1}\{\sum^n_{i=1}
[2\lambda_i g^{ii}g^{ii}+
2(1-\delta_{i\gamma})(\lambda_i+\lambda_\gamma) g^{ii}g^{\gamma\gamma}]u_{ii\gamma}^2
-\frac{1+\epsilon}{A+\Delta\cdot}g^{\gamma\gamma}(\sum^n_{i=1}u_{ii\gamma})^2\}.
\end{aligned}
\end{equation}
The constant $\epsilon, \delta, A$ will be decided later. We consider the following quadratic form with redefined terms $u_{ii\gamma}=u_{ii\gamma}-\frac{f_\gamma}f\lambda_i$ which will give us a convenient equality $\sum^n_{i=1}g^{ii}u_{ii\gamma}=0$ for all $1\le i,\gamma\le n$,
\begin{align}\label{eq_Q_gamma}
Q_{\gamma}=\sum^n_{i=1}
[2\lambda_i g^{ii}g^{ii}+
2(1-\delta_{i\gamma})(\lambda_i+\lambda_\gamma) g^{ii}g^{\gamma\gamma}]u_{ii\gamma}^2
-\frac{1+\epsilon}{A+\Delta\cdot}g^{\gamma\gamma}(\sum^n_{i=1}u_{ii\gamma})^2.
\end{align}

\subsection {Proof to Lemma \ref{lemma_Trace_Jacobi} when $n=3, \Theta\ge\frac\pi2$}
\begin{proof}

Let $c(3)>0$ be some small constant. We want to show the nonnegativity of $Q_\gamma$ for all $1\le \gamma\le 3$.  We will discuss according to the value of $\lambda_3$.\\

\textbf{Case 1.} $\lambda_3<0$. We have
\begin{equation}\label{eq_sigma2}
    \frac\pi2-\arctan(\lambda_1/f)+\frac\pi2-\arctan(\lambda_2/f)\ge \frac\pi2+\arctan(\lambda_3/f).
\end{equation}
Thus $\sum^3_{i=1}\frac1{\lambda_i}\le0$. Let $\lambda_2=(1+\nu)|\lambda_3|$, $\lambda_1=(1+\nu_1)|\lambda_3|$, then $\nu_1\ge \max(\nu^{-1},1)$.  We also have the following inequality since $\sum^3_{i=1}g^{ii}u_{ii\gamma}=0$,
$$
\sum^3_{i=1}
\lambda_i g^{ii}g^{ii}u_{ii\gamma}^2\ge 0.
$$
Denote $\tau_{ij}=g_{ii}/g_{jj}$, the negative part of (\ref{eq_Q_gamma}) can be reduced slightly,
\begin{equation}\label{eq_case_22}
(\sum^3_{i=1}u_{ii\gamma})^2
= [(1-\tau_{31})u_{11\gamma}+(1-\tau_{32})u_{22\gamma}]^2.
\end{equation}
When $\gamma=3$, we consider the quadratic form
\begin{equation}\label{eq_hatQ3}
\hat Q_3= \sum^2_{i=1}
2(\lambda_i+\lambda_3)(A+\Delta\cdot) g^{ii}u_{ii3}^2
-(1+\epsilon)[\sum^2_{i=1}(1-\tau_{3i})u_{ii3}]^2.
\end{equation}
Our goal is to prove $\hat Q_3\ge0$ for some $\epsilon=10^{-1}$. By the diagonalize lemma (see the appendix) we have
\begin{align*}
&1-(1+\epsilon)\sum^2_{i=1}\frac{(1-\tau_{3i})^2g_{ii}}{2(\lambda_i+\lambda_3)\Delta\cdot}\ge 1-(1+\epsilon)\frac1{2\Delta\cdot}\sum^2_{i=1}(\lambda_i-\lambda_3)(1-\frac{\lambda_3^2}{\lambda_i^2})\\
&\ge1-(1+\epsilon)\frac12[1+\frac{1}{1+\nu+\nu^{-1}}(3-\frac{2+\nu}{(1+\nu)^2}-\frac{2+\nu^{-1}}{(1+\nu^{-1})^2})]\\
&\ge 1-\frac56(1+\epsilon)\ge 0.
\end{align*}
Thus for a smaller $\epsilon(3)>0$, there exist constants $\hat\delta(3),\delta(|f|_\infty)>0$ such that
\begin{equation}\label{eq hatQ3_delta}
Q_3\ge\hat\delta(3) \sum^2_{i=1}(\lambda_i+\lambda_3)g^{ii}g^{33}u^2_{ii3} \ge3\delta \sum^3_{i=1}(\Delta\cdot)^{-1}g^{ii}u^2_{ii3}.
\end{equation}
When $\gamma=1$, let $0<\hat \epsilon<1$, we split (\ref{eq_Q_gamma}) into two parts, and consider the following two quadratic forms,
\begin{equation}\label{eq_hatQ1}
    \begin{aligned}
    &\hat Q^{(1)}_1= (A+\Delta\cdot)[(1+\hat \epsilon)\lambda_1 g^{11}u_{111}^2+2\lambda_2g^{22}u_{221}^2]
-(1+\epsilon)[\sum^2_{i=1}(1-\tau_{3i})u_{ii1}]^2,\\
&\hat Q^{(2)}_1= (1-\hat \epsilon)\lambda_1g^{11}g^{11}u_{111}^2+2(\lambda_2+\lambda_1\tau_{21})g^{22} g^{22}u_{221}^2
+2[\lambda_3+(\lambda_1+\lambda_3)\tau_{31}]g^{33}g^{33}u_{331}^2.
    \end{aligned}
\end{equation}
We need to determine $\hat \epsilon$, $\epsilon$ such that the following two inequalities holds,
\begin{align*}
&1-(1+\epsilon)[\frac1{1+\hat\epsilon}\frac{(1-\tau_{31})^2g_{11}}{\lambda_1(A+\Delta\cdot)}+\frac{(1-\tau_{32})^2g_{22}}{2\lambda_2(A+\Delta\cdot)}]\\
&\ge1-(1+\epsilon)\frac1{\Delta\cdot}[\frac1{1+\hat\epsilon}\lambda_1(1-\frac{\lambda_3^2}{\lambda_1^2})+\frac12\lambda_2(1-\frac{\lambda_3^2}{\lambda_2^2})]\\
&\ge1-\frac{1+\epsilon}{1+\nu+\nu_1}[\frac{1+\nu_1}{1+\hat\epsilon}+\frac{1+\nu}2(1-\frac1{(1+\nu)^2})]\\
&\ge 1-\frac{1+\epsilon}{1+\hat\epsilon}[1-\frac\nu{1+\nu+\nu_1}(1-\frac{1+\hat\epsilon}2\frac{2+\nu}{1+\nu})]\ge0.
\end{align*}
\begin{align*}
&[2 (1-\hat \epsilon)^{-1}-1]\frac1{\lambda_1}+[(\lambda_2+\lambda_1\tau_{21})^{-1}-\frac1{\lambda_2}]
+[\lambda_3+(\lambda_1+\lambda_3)\tau_{31}]^{-1}-\frac1{\lambda_3}\\
&\le \frac1{\lambda_1}[\frac{1+\hat \epsilon}{1-\hat \epsilon}-\frac{\lambda_1}{\lambda_1+\lambda_2}]+\frac1\lambda_3[\frac1{1+\frac{\lambda_3}{\lambda_1}
+\frac{\lambda_3^2}{\lambda_1^2}}-1]\\
&\le \frac1{\lambda_1}[\frac{1+\hat \epsilon}{1-\hat \epsilon}-\frac{\lambda_1}{\lambda_1+\lambda_2}-\frac{\lambda_1+\lambda_3}{\frac54\lambda_1+\lambda_3}]\le \frac1{\lambda_1}[\frac{1+\hat \epsilon}{1-\hat \epsilon}-\frac{7}{6}] \le 0.
\end{align*}
In the second inequality we assume that $\lambda_3+(\lambda_1+\lambda_3)\tau_{31}<0$. Now the first inequality holds with $1>\hat\epsilon=2\epsilon= 2\times10^{-2}$, the second inequality allows $\hat\epsilon\le\frac1{13}$. Thus for a smaller $\epsilon(3)>0$,  the first inequality implies $\hat Q_1^{(1)}\ge\hat\delta(3)(\sum^2_{i=1}\lambda_ig^{ii}g^{ii}u^2_{ii\gamma})(A+\Delta\cdot)$. Recall $\sum^3_{i=1}g^{ii}u_{ii\gamma}=0$, then for $A=A(|f|_\infty)$, $\delta=\delta(|f|_\infty)>0$,
\begin{equation}\label{eq hatQ1_delta}
Q_1\ge\frac14\hat\delta(3) \sum^3_{i=1}|\lambda_i|g^{ii}g^{ii}u^2_{ii1} \ge3\delta \sum^3_{i=1}(A+\Delta\cdot)^{-1}g^{ii}u^2_{ii1}.
\end{equation}

\noindent When $\gamma=2$, we consider the following quadratic forms with new parameters $\hat \epsilon$, $\epsilon$,
\begin{equation}\label{eq hatQ2}
\begin{aligned}
&\hat Q^{(1)}_2= (A+\Delta\cdot)[2\lambda_1 g^{11}u_{112}^2+(1-\hat\epsilon)\lambda_2g^{22}u_{222}^2]
-(1+\epsilon)[\sum^2_{i=1}(1-\tau_{3i})u_{ii2}]^2,\\
&\hat Q^{(2)}_2=2(\lambda_1+\lambda_2\tau_{12})g^{11}g^{11}u_{112}^2+(1+\hat \epsilon)\lambda_2g^{22} g^{22}u_{222}^2
+2[\lambda_3+(\lambda_2+\lambda_3)\tau_{32}]g^{33}g^{33}u_{332}^2.
\end{aligned}
\end{equation}
We want to determine the sign of $\hat Q^{(1)}_2,\hat Q^{(2)}_2$. We calculate that
\begin{align*}
P_1=&1-(1+\epsilon)[\frac1{2}\frac{(1-\tau_{31})^2g_{11}}{\lambda_1(A+\Delta\cdot)}+\frac1{1-\hat\epsilon}\frac{(1-\tau_{32})^2g_{22}}{\lambda_2(A+\Delta\cdot)}]\\
&\ge1-(1+\epsilon)\frac1{\Delta\cdot}[\frac{1}{2}\lambda_1(1-\frac{\lambda_3^2}{\lambda_1^2})+\frac{1}{1-\hat\epsilon}\lambda_2(1-\frac{\lambda_3^2}{\lambda_2^2})].
\end{align*}
\begin{align*}
P_2&=[\lambda_1+\lambda_2\tau_{12}]^{-1}-\frac1\lambda_1+(\frac2{1+\hat\epsilon}-1)\frac1{\lambda_2}
+[\lambda_3+(\lambda_2+\lambda_3)\tau_{32}]^{-1}-\frac1{\lambda_3}\\
&\le 
\frac1{\lambda_2}\frac{1-\hat\epsilon}{1+\hat\epsilon}-\frac1\lambda_2\frac{1+\frac{\lambda_3}{\lambda_2}}{1+\frac{\lambda_3}{\lambda_2}
+\frac{\lambda_3^2}{\lambda_2^2}}.
\end{align*}
We choose $\hat \epsilon=\frac{\lambda_3^2}{\lambda_2^2}$, recall $\lambda_1=(1+\nu_1)|\lambda_3|\ge 2|\lambda_3|$, $\lambda_2=(1+\nu)|\lambda_3|$,
\begin{align*}
P_1&\ge1-\frac{1+\epsilon}{1+\nu+\nu_1}[\frac{1}{2}(1+\nu_1)+1+\nu
-\frac1{2(1+\nu_1)}]\\
&\ge 1-(1+\epsilon)(1-\frac12\frac{\nu_1-1}{1+\nu+\nu_1}-\frac1{2(1+\nu_1)})
\ge 1-\frac{11}{12} (1+\epsilon)\ge 0,\\
P_2&= -\frac1{\lambda_2}[\frac{\nu(1+\nu)}{1+\nu(1+\nu)}-\frac{\nu(2+\nu)}{1+(1+\nu)^2}]
\le0,
\end{align*}
if we choose $\epsilon=\epsilon(3)$ small. Similar to $\gamma=1$, for a smaller $\epsilon=\epsilon(3)>0$ we have
\begin{equation}\label{eq hatQ2_delta}
Q_2\ge \hat\delta(3) [\lambda_1g^{11}g^{22}u^2_{112}+\lambda_2(1-\frac{\lambda_3^2}{\lambda_2^2})g^{22}g^{22}u^2_{222}]\ge 3\delta(A+\Delta\cdot)^{-1}\sum^3_{i=1}g^{ii}u^2_{ii2}.
\end{equation}
Finally, we return to the original $u_{ii\gamma}$ in $Q_\gamma$ and get the trace Jacobi inequality by (\ref{eq_35}) (\ref{eq_Q_gamma}) (\ref{eq hatQ3_delta}) (\ref{eq hatQ1_delta}) (\ref{eq hatQ2_delta}),
$$
   \begin{aligned}
&\Delta_g|A+\Delta\cdot|-(1+\frac\epsilon2)\frac{|\nabla_g|A+\Delta\cdot||^2}{A+\Delta\cdot}\\ &\ge\delta|A+\Delta\cdot|^{-1}\sum^3_{\gamma,i=1}g^{ii}u^2_{ii\gamma}+ \frac{\Delta f} f\sum^3_{i=1}g^{ii}\lambda_{i}-{C_f}T_{g^{-1}}|A+\Delta\cdot|.
   \end{aligned}
$$
\textbf{Case 2.} $\lambda_3\ge0$. We leave the proof to the next subsection for all $n\ge3$.
\end{proof}

\subsection{Proof to Lemma \ref{lemma_Trace_Jacobi} when a) $\Theta\ge\frac\pi2(n-2), n\ge4$; b) $\lambda_n \ge0,$ $n\ge 3$}
\begin{proof}
When $n\ge4$ the situation may get a little better. The argument is based on the special case $n=3$.\\
\textbf{Case 1.} $\lambda_n<0$. We assume that $n\ge4$. We still have the equality
$$
\sum^n_{i=1}
2\lambda_i g^{ii}g^{ii}u_{ii\gamma}^2\ge 0,
$$
since $\sum^n_{i=1}\frac1{\lambda_i}\le 0$ (Wang-Yuan \cite{WY11}, Lemma 2.1), and $\sum^n_{i=1}g^{ii}u_{ii\gamma}=0$ for all $1\le \gamma\le n$. We also have a basic fact that $\lambda_i\ge (n-i)|\lambda_n|$, $1\le i\le n-1$ (Wang-Yuan \cite{WY11}, Lemma 2.1), in particular we get $\lambda_{n-2}\ge 2|\lambda_n|$. Let $\lambda_i=(1+\nu_i)|\lambda_n|$, $\tau_{ij}=g_{ii}/g_{jj}$, $1\le i,j\le n$. We have
\begin{equation}\label{eq_case_22n}
(\sum^n_{i=1}u_{ii\gamma})^2
= [\sum^{n-1}_{i=1}(1-\tau_{ni})u_{ii\gamma}]^2.
\end{equation}
When $\gamma=n$, we consider the quadratic form
\begin{equation}\label{eq_hatQ3n}
\hat Q_n= \sum^{n-1}_{i=1}
2(\lambda_i+\lambda_n)(A+\Delta\cdot) g^{ii}u_{iin}^2
-(1+\epsilon)[\sum^{n-1}_{i=1}(1-\tau_{ni})u_{iin}]^2.
\end{equation}
Our goal is to prove $\hat Q_n\ge0$ for some $\epsilon=\epsilon(n)>0$. We have
\begin{align*}
&1-(1+\epsilon)\sum^{n-1}_{i=1}\frac{(1-\tau_{ni})^2g_{ii}}{2(\lambda_i+\lambda_n)\Delta\cdot}\ge 1-(1+\epsilon)\frac1{2\Delta\cdot}\sum^{n-1}_{i=1}(\lambda_i-\lambda_n)(1-\frac{\lambda_n^2}{\lambda_i^2})\\
&\ge1-(1+\epsilon)[1-\frac{\frac12\sum^{n-1}_{i=1}\nu_i-1}{n-2+\sum^{n-1}_{i=1}\nu_i}]
\ge 1-(1+\epsilon)(1-\frac1{6})\ge 0.
\end{align*}
Thus for a smaller $\epsilon(n)>0$ and constants $\hat\delta=\hat\delta(n), \delta=\delta(n,|f|_\infty)>0$ such that
\begin{equation}\label{eq hatQ3_deltan}
Q_\gamma\ge\hat\delta(n) \sum^{n-1}_{i=1}(\lambda_i+\lambda_n)g^{ii}g^{nn}u^2_{iin} \ge3\delta \sum^n_{i=1}(\Delta\cdot)^{-1}g^{ii}u^2_{iin}.
\end{equation}
When $\gamma\le n-2$, let $0<\hat \epsilon<1$, we split (\ref{eq_Q_gamma}) into two parts, and consider the following two quadratic forms,
\begin{equation}\label{eq_hatQ1n}
    \begin{aligned}
\hat Q^{(1)}_\gamma&= (A+\Delta\cdot)[(1+\hat \epsilon)\lambda_\gamma g^{\gamma\gamma}u_{\gamma\gamma\gamma}^2+2\sum^{n-1}_{i\neq\gamma}\lambda_i g^{ii}u_{ii\gamma}^2]
-(1+\epsilon)[\sum^{n-1}_{i=1}(1-\tau_{ni})u_{ii\gamma}]^2,\\
\hat Q^{(2)}_\gamma&= (1-\hat \epsilon)\lambda_\gamma g^{\gamma\gamma}g^{\gamma\gamma}u_{\gamma\gamma\gamma}^2+2\sum^{n-1}_{i\neq\gamma}(\lambda_i+\lambda_\gamma\tau_{i\gamma})g^{ii} g^{ii}u_{ii\gamma}^2\\
&+2[\lambda_n+(\lambda_\gamma+\lambda_n)\tau_{n\gamma}]g^{nn}g^{nn}u_{nn\gamma}^2.
    \end{aligned}
\end{equation}
Assume that $\lambda_n+(\lambda_\gamma+\lambda_n)\tau_{n\gamma}<0$, we need to determine $\hat \epsilon<1$, $\epsilon$ such that the following two inequalities holds,
\begin{align*}
&1-(1+\epsilon)[\frac1{1+\hat\epsilon}\frac{(1-\tau_{n\gamma})^2g_{\gamma\gamma}}{\lambda_\gamma(A+\Delta\cdot)}
+\sum^{n-1}_{i\neq\gamma}\frac{(1-\tau_{ni})^2g_{ii}}{2\lambda_i(A+\Delta\cdot)}]\\
&\ge1-(1+\epsilon)\frac1{\Delta\cdot}[\frac1{1+\hat\epsilon}\lambda_\gamma(1-\frac{\lambda_n^2}{\lambda_\gamma^2})
+\frac12\sum^{n-1}_{i\neq\gamma}\lambda_i(1-\frac{\lambda_n^2}{\lambda_i^2})]\\
&\ge1-\frac{1+\epsilon}{n-2+\sum^{n-1}_{i=1}\nu_i}[\frac{1+\nu_\gamma}{1+\hat\epsilon}(1-\frac1{(1+\nu_\gamma)^2})
+\sum^{n-1}_{i\neq\gamma}\frac{1+\nu_i}2(1-\frac1{(1+\nu_i)^2})]\\
&\ge 1-\frac{1+\epsilon}{1+\hat\epsilon}[1-\frac1{n-2+\sum^{n-1}_{i=1}\nu_i}(n-3+\frac1{1+\nu_\gamma}
+\frac12\sum^{n-1}_{i\neq\gamma}\frac{(1-\hat\epsilon)\nu_i^2-2\hat\epsilon\nu_i}{1+\nu_i})]\\
&\ge 1-\frac{1+\epsilon}{1+\hat\epsilon}[1-\frac1{n-2+\sum^{n-1}_{i=1}\nu_i}(n-4+\frac1{1+\nu_\gamma}
+\frac12\sum^{n-2}_{i\neq\gamma}\frac{(1-\hat\epsilon)\nu_i^2-2\hat\epsilon\nu_i}{1+\nu_i})]\ge0.
\end{align*}
\begin{align*}
&[2 (1-\hat \epsilon)^{-1}-1]\frac1{\lambda_\gamma}+\sum^{n-1}_{i\neq\gamma}[(\lambda_i+\lambda_\gamma\tau_{i\gamma})^{-1}-\frac1{\lambda_i}]
+[\lambda_n+(\lambda_\gamma+\lambda_n)\tau_{n\gamma}]^{-1}-\frac1{\lambda_n}\\
&\le \frac1{\lambda_\gamma}[\frac{1+\hat \epsilon}{1-\hat \epsilon}-\frac{\lambda_\gamma}{\lambda_\gamma+\lambda_{n-1}}]+\frac1\lambda_n[\frac1{1+\frac{\lambda_n}{\lambda_\gamma}
+\frac{\lambda_n^2}{\lambda_\gamma^2}}-1]\\
&\le \frac1{\lambda_\gamma}[\frac{1+\hat \epsilon}{1-\hat \epsilon}-\frac{\lambda_\gamma}{\lambda_\gamma+\lambda_{n-1}}-\frac{\lambda_\gamma+\lambda_n}{\frac54\lambda_\gamma+\lambda_n}]\le \frac1{\lambda_\gamma}[\frac{1+\hat \epsilon}{1-\hat \epsilon}-\frac{7}{6}] \le 0.
\end{align*}
The first inequality holds with $\hat\epsilon=2\epsilon =10^{-2}$ since $\nu_{n-2}\ge1$, the second inequality allows $\hat\epsilon\le\frac1{13}$, where we use $\lambda_{n-2}\ge 2|\lambda_n|$. Thus for a smaller $\epsilon(n)>0$, we have constants $\hat\delta(n)$, $A=A(n,|f|_\infty)$, $\delta=\delta(n,|f|_\infty)>0$ such that,
\begin{equation}\label{eq hatQ1_deltan}
Q_\gamma\ge\hat\delta(n) \sum^{n-1}_{i=1}|\lambda_i|g^{ii}g^{ii}u^2_{ii1} \ge3\delta \sum^n_{i=1}(A+\Delta\cdot)^{-1}g^{ii}u^2_{ii1}.
\end{equation}

\noindent When $\gamma=n-1$, we consider the following quadratic forms with new $\hat \epsilon$, $\epsilon$,
\begin{equation}\label{eq hatQ2n}
\begin{aligned}
\hat Q^{(1)}_\gamma&= (A+\Delta\cdot)[2\sum^{n-1}_{i\neq\gamma}\lambda_i g^{ii}u_{ii\gamma}^2+(1-\hat\epsilon)\lambda_\gamma g^{\gamma\gamma}u_{\gamma\gamma\gamma}^2]
-(1+\epsilon)[\sum^{n-1}_{i=1}(1-\tau_{ni})u_{ii\gamma}]^2,\\
\hat Q^{(2)}_\gamma&=2\sum^{n-1}_{i\neq\gamma}(\lambda_i+\lambda_\gamma\tau_{i\gamma})g^{ii}g^{ii}u_{ii\gamma}^2+(1+\hat \epsilon)\lambda_\gamma g^{\gamma\gamma} g^{\gamma\gamma}u_{\gamma\gamma\gamma}^2
+2[\lambda_n+(\lambda_\gamma+\lambda_n)\tau_{n\gamma}]g^{nn}g^{nn}u_{nn\gamma}^2.
\end{aligned}
\end{equation}
We want to determine the sign of $\hat Q^{(1)}_\gamma,\hat Q^{(2)}_\gamma$. Let $\hat \epsilon=\frac{\lambda_n^2}{\lambda_\gamma^2}$, we need to choose $\epsilon$ such that the following two inequalities hold,
\begin{align*}
P_1&=1-(1+\epsilon)[\frac1{2}\sum^{n-1}_{i\neq\gamma}\frac{(1-\tau_{ni})^2g_{ii}}{\lambda_i(A+\Delta\cdot)}
+\frac1{1-\hat\epsilon}\frac{(1-\tau_{n\gamma})^2g_{\gamma\gamma}}{\lambda_\gamma(A+\Delta\cdot)}]\\
&\ge1-(1+\epsilon)\frac1{A+\Delta\cdot}[\frac{1}{2}\sum^{n-1}_{i\neq\gamma}\lambda_i(1-\frac{\lambda_n^2}{\lambda_i^2})
+\frac{1}{1-\hat\epsilon}\lambda_\gamma(1-\frac{\lambda_n^2}{\lambda_\gamma^2})]\\
&\ge1-(1+\epsilon)\frac1{n-2+\sum^{n-1}_{i=1}\nu_i}[\frac{3}{4}\sum^{n-1}_{i=1}\nu_i+\frac34(n-2)]\ge 1-\frac34({1+\epsilon})\ge0.
\end{align*}
\begin{align*}
P_2&=\sum^{n-1}_{i\neq\gamma}[(\lambda_i+\lambda_\gamma\tau_{i\gamma})^{-1}-\frac1\lambda_i]+(\frac2{1+\hat\epsilon}-1)\frac1{\lambda_\gamma}
+[\lambda_n+(\lambda_\gamma+\lambda_n)\tau_{n\gamma}]^{-1}-\frac1{\lambda_n}\\
&\le\frac1{\lambda_\gamma}\frac{1-\hat\epsilon}{1+\hat\epsilon}
-\frac1\lambda_\gamma[\frac{1+\frac{\lambda_n}{\lambda_\gamma}}{1+\frac{\lambda_n}{\lambda_\gamma}
+\frac{\lambda_n^2}{\lambda_\gamma^2}}]
=\frac1{\lambda_\gamma}[\frac{\nu_\gamma(2+\nu_\gamma)}{1+(1+\nu_\gamma)^2}-\frac{\nu_\gamma(1+\nu_\gamma)}{1+\nu_\gamma+\nu_\gamma^2}]\le 0.
\end{align*}
We can choose $\epsilon=\frac14$. Thus for a smaller $\epsilon=\epsilon(n)$ and $\delta=\delta(n,|f|_\infty)>0$,
\begin{equation}\label{eq hatQ2_deltan}
Q_\gamma\ge \hat\delta(n) [\sum^{n-1}_{i\neq\gamma}\lambda_ig^{ii}g^{\gamma\gamma}u^2_{ii\gamma}
+\lambda_\gamma(1-\frac{\lambda_n^2}{\lambda_\gamma^2})g^{\gamma\gamma}g^{\gamma\gamma}u^2_{\gamma\gamma\gamma}]\ge 3\delta(A+\Delta\cdot)^{-1}\sum^n_{i=1}g^{ii}u^2_{ii\gamma}.
\end{equation}
By (\ref{eq hatQ3_deltan}) (\ref{eq hatQ1_deltan}) (\ref{eq hatQ2_deltan}) we get for $\epsilon=\epsilon(n)$, $A=A(n,|f|_\infty)>0$, $\delta=\delta(n,|f|_\infty)>0$,
\begin{equation}\label{eq_QGamma_nCri}
    Q_\gamma\ge3\delta(A+\Delta\cdot)^{-1}\sum^{n}_{i=1}g^{ii}u^2_{ii\gamma}.
\end{equation}
\textbf{Case 2.} $\lambda_n\ge0$. Here we prove the general convex case for all dimension $n\ge 3$.\\
Let $m=m(n)<1$, $M=M(n)>1$ be some small/large constants. Recall we have $\sum^n_{i=1} g^{ii}u_{ii\gamma}=0$ by reducing the derivative of $f$ to $u_{ii\gamma}$, we consider the quadratic form
$$
Q_\gamma=\sum^{n}_{i=1}
2\lambda_i(A+\Delta\cdot) g^{ii}u_{ii\gamma}^2
-(1+\epsilon)(\sum^{n}_{i=1}u_{ii\gamma})^2.
$$
If $\lambda_{n}\ge m(n)$, to verify $Q_\gamma\ge0$ we only need $\epsilon=\frac12$, $A=A(n,|f|_\infty)$ large such that
\begin{align*}
&1-(1+\epsilon)\sum^{n}_{i=1}\frac{g_{ii}}{2\lambda_i(A+\Delta\cdot)}\ge 1-\frac{1+\epsilon}{2(A+\Delta\cdot)}\sum^{n}_{i=1}(\lambda_i+\frac1m |f|^2_{\infty})\ge 1-\frac{1+\epsilon}2\ge0.
\end{align*}
If $\lambda_{1}\le M(n)$, we verify directly that for $A=nM$,
\begin{align*}
&(\sum^{n}_{i=1}u_{ii\gamma})^2=[\sum^{n}_{i=1}(u_{ii\gamma}-f^2g^{ii}u_{ii\gamma})]^2\le nM\sum^{n}_{i=1}\lambda_ig^{ii}u^2_{ii\gamma}
\le \sum^{n}_{i=1}
\lambda_i(A+\Delta\cdot) g^{ii}u_{ii\gamma}^2.
\end{align*}
Otherwise $\lambda_n<m(n)$, $\lambda_1>M(n)$, we split $\{\lambda_i\}^n_{i=1}$ into two parts,
$$
\left\{
\begin{array}{ll}
 (\lambda_1,\cdots,\lambda_k)    & \lambda_k\ge m   \\
(\lambda_{k+1},\cdots ,\lambda_{n})   &  \lambda_{k+1}< m     \\
\end{array} \right.
,
$$
where $k\le n-1$. We have the following relation,
\begin{equation}\label{eq_case_22}
\begin{aligned}
\sum^k_{i=1}\frac {u^2_{ii\gamma}}{(f^2+\lambda_i^2)^2}&\ge \frac1n(\sum^k_{i=1}\frac{ u_{ii\gamma}}{f^2+\lambda_i^2})^2=\frac1n(\sum^n_{i=k+1}\frac{u_{ii\gamma}}{f^2+\lambda_i^2})^2\\
& \ge \frac{1}{n^2f^4}(\sum^n_{i=k+1}u_{ii\gamma})^2-n^2m\sum^n_{i=k+1}\lambda_ig^{ii}u^2_{ii\gamma}.
\end{aligned}
\end{equation}
We choose $A\ge4n^6|f|^4_\infty/m$, by (\ref{eq_case_22}) we have
$$
\sum^k_{i=1}\frac12\lambda_i(A+\Delta\cdot)g^{ii}u^2_{ii\gamma}
+\sum^n_{i=k+1}2\lambda_i(A+\Delta\cdot)g^{ii}u^2_{ii\gamma}\ge {n^2}(\sum^n_{i=k+1}u_{ii\gamma})^2.
$$
We get
$$
Q_\gamma \ge \sum^{k}_{i=1}
\frac32\lambda_i(A+\Delta\cdot) g^{ii}u_{ii\gamma}^2 +n^2(\sum^n_{i=k+1}u_{ii\gamma})^2 -(1+\epsilon)(\sum^n_{i=1}u_{ii\gamma})^2.
$$
Now we choose $A=\frac1m(4n^6|f|_\infty^4+n|f|_\infty^2)$, $\epsilon=\frac1{6}$,
\begin{align*}
1-(1+\epsilon)(\sum^{k}_{i=1}\frac{2g_{ii}}{3\lambda_i(A+\Delta\cdot)}+\frac1{n^2})&\ge 1-\frac23\frac{1+\epsilon}{(A+\Delta\cdot)}\sum^{k}_{i=1}(\lambda_i+\frac1m |f|^2_{\infty})+\frac{1+\epsilon}{n^2}\\
&\ge 1-\frac56(1+\epsilon)\ge0.
\end{align*}
Thus for a smaller $\epsilon(n)$, and $A=A(n,|f|_\infty)$, $\delta=\delta(n,|f|_\infty)>0$,
\begin{equation}\label{eq hatQ2_deltanCovex}
Q_\gamma\ge \delta(n)\sum^n_{i=1}|\lambda_i|g^{ii}g^{ii}u^2_{ii\gamma}\ge3\delta(A+\Delta\cdot)^{-1}\sum^{n}_{i=1}g^{ii}u^2_{ii\gamma}.
\end{equation}
Finally, we return to the original $u_{ii\gamma}$ in $Q_\gamma$ and get the trace Jacobi inequality by (\ref{eq_35}) (\ref{eq_QGamma_nCri}) (\ref{eq hatQ2_deltanCovex}) that
$$
   \begin{aligned}
&\Delta_g|A+\Delta\cdot|-(1+\frac\epsilon2)\frac{|\nabla_g|A+\Delta\cdot||^2}{A+\Delta\cdot}\\ &\ge3\delta|A+\Delta\cdot|^{-1}\sum^n_{\gamma,i=1}g^{ii}u^2_{ii\gamma}+ \frac{\Delta f} f\sum^n_{i=1}g^{ii}\lambda_{i}-{C_f}T_{g^{-1}}|A+\Delta\cdot|.
   \end{aligned}
$$
\end{proof}

\section{A Prior Interior Hessian Estimates}

\subsection {Proof to Theorem \ref{Thm_Upbd_critical}}
\begin{proof}
\textbf{Step 1.} Up to a scale we will assume $f\ge1$. We prove a mean value inequality similar to Proposition \ref{prop_Meanvalue}, to eliminate the term $\Delta f$. Let $U\subset \mathcal{M}$ be a open set, $\varphi(x)\in C_0^\infty(U)$ be a nonnegative function. Let $A\ge 3$, by (\ref{eq_Trace_Jacobin1}) we have
\begin{align*}
&\int_{U}\ln|A+\Delta\cdot| \Delta_g \varphi dv_g\ge\frac\epsilon2\int_{U}|\nabla_g\ln|A+\Delta\cdot||^2\varphi dv_g+
\delta\int_{U}\frac{\sum^n_{\gamma,i=1}g^{ii}u^2_{ii\gamma}}{|A+\Delta\cdot|^2}\varphi dv_g \\
&-n\int_{U}\frac{|D\varphi Df|}{A+\Delta\cdot} dv_g-
\int_{U}DfD(\frac{\sum^n_{i=1}g^{ii}\lambda_{i}}{ f|A+\Delta\cdot|}V)\varphi dx-C_f\int_{U}T_{g^{-1}}\varphi dv_g.
\end{align*}
where $V=\Pi^n_{i=1}\sqrt{1+\lambda_i^2}$. Pointwisely we can choose a coordinate system such that $D^2u$ is diagonal, and
\begin{align*}
&|Df D V|=|\sum^n_{\gamma,i=1}\frac{\lambda_i\lambda_{i,\gamma}}{1+\lambda_i^2}V D_\gamma f|
\le \frac1{4n}\delta\frac{\sum^n_{\gamma,i=1}g^{ii}u^2_{ii\gamma}}{A+\Delta\cdot}V+C_f|A+\Delta\cdot| V,\\
&|Df D\Delta\cdot|\le \frac1{4n}\epsilon|\nabla_g|A+\Delta\cdot||^2+C_f|A+\Delta\cdot|^2,
\end{align*}
\begin{align*}
&|Df D\sum^n_{i=1}g^{ii}\lambda_i|\le C(n)\sum^n_{\gamma,i=1}|\frac{|\lambda_{i,\gamma}|+C_f}{1+\lambda_i^2}D_\gamma f|
\le \frac1{4}\delta\frac{\sum^n_{\gamma,i=1}g^{ii}u^2_{ii\gamma}}{A+\Delta\cdot}+C_f|A+\Delta\cdot| .
\end{align*}
Let $b=\ln|A+\Delta\cdot|>1$, notice that $f\ge1,|\sum^n_{i=1}g^{ii}\lambda_{i}|\le n$, we get
\begin{equation}\label{eq_11}
\int_{U}b\Delta_g\varphi dv_g\ge
-n\int_{U}\frac{|D\varphi Df|}{A+\Delta\cdot} dv_g-C_f\int_{U} \varphi dv_g.
\end{equation}
Suppose that $Du(0)=0$, let $r=\sqrt{f^2|x|^2+|Du|^2}$, $S_\rho(0)=\{x\in \mathcal{M}|r<\rho\}\subset U$. Let $\chi(t)\in C^\infty(\mathbb{R})$ be a non-decreasing function such that $\chi(t)=0$ when $t\le0$, $\chi(t)=1$ when $t>\mu$ for some small constant $\mu>0$. Let $\rho<1$, we choose
$$
\varphi(r)=\int^{+\infty}_rt\chi(\rho-t)dt.
$$
We have $0\le\varphi\le \rho^2\chi(\rho-r)$, $|D\varphi(x)|\le C(n)\rho|A+\Delta\cdot|\chi(\rho-r)$, $|\nabla_g r|\le C_f$, and
$$
\Delta_g \varphi(r)\le (-n+C_fr-f\Delta_gf|x|^2)\chi(\rho-r)+(r+C_fr^2)\chi'(\rho-r).
$$
Since $b>1$, (\ref{eq_11}) becomes
\begin{align*}
&\int_{U}b[n\chi(\rho-t)-\rho\chi'(\rho-r)] dv_g\le
C_f[\rho\int_{U} b\chi(\rho-r)dv_g+\rho^2\int_{U} b\chi'(\rho-r)dv_g].
\end{align*}
Multiply the above inequality by $\rho^{-n-1}$,
$$
-(1+C_f\rho)\frac{d}{d\rho}\ln [\frac{1}{\rho^n}\int_{U}b\chi(\rho-t)dv_g]\le (n+1)C_f.
$$
Choose $\rho_0\in(0,1)$, integrate the above inequality over $(\sigma,\rho_0)$ for arbitrarily small constants $0<\mu<\sigma$, we get
$$
\sigma^{-n}\int_{S_{\sigma-\mu}(0)}bdv_g \le C_f\rho_0^{-n}\int_{S_{\rho_0}(0)}bdv_g.
$$
Let ${\mathcal{B}}_\rho(0)\subset U$ be the geodesic ball on $\mathcal{M}$ with radius $\rho$, we have $ \mathcal{B}_\frac{\rho}{|f|_\infty}(0)\subset S_{\rho}(0)\subset\mathcal{B}_{\rho|f|_\infty}(0)$. For general $0<\rho<1$ we get
\begin{equation}\label{eq_meanValue}
b(0)\le  C_f\rho^{-n}\int_{\mathcal{B}_\rho(0)}bdv_g.
\end{equation}
\textbf{Step 2.} We follow the technique developed in Wang-Yuan \cite{WY11} to get the final result. We substitute $U$ by ${B}_2(0)\in \mathbb{R}^n$, and choose the smooth test function $\varphi$ such that $\varphi\ge 0$ in ${B}_2(0)$, $\varphi=1$ in ${B}_1(0)$, $|D\varphi|<2$. Repeat the proof at the beginning of {Step 2.} we get for $C_f=C_f(n,||f||_{C^{0,1}({ B}_3(0))})$

$$
\int_{{{B}}_2(0)}\Delta_gb\varphi^2 dv_g-\frac\epsilon4\int_{{{B}}_2(0)}{|\nabla_gb|^2}\varphi^2 dv_g\ge
-C_f\int_{{B}_2(0)} \varphi^2dv_g.
$$
Recall $T_{g^{-1}}=\sum^n_{i=1}g^{ii}$, we get the gradient estimate for $b$,
\begin{equation}\label{eq_14}
\int_{{{B}}_2(0)}{|\nabla_gb|^2}\varphi^2dv_g\le
C_f\int_{{B}_2(0)}  \varphi^2dv_g.
\end{equation}
We use mean value inequality and Sobolev inequality (Proposition \ref{prop_Sobolev}) to bound $b(0)$,
\begin{equation}\label{eq_151}
\begin{aligned}
b(0)&\le C_*(\int_{{B}_1(0)}b^{\frac n{n-2}}dv_g)^\frac{n-2}{n}
    \le C_*[\int_{B_2(0)}(\varphi^{n+1} b^\frac12)^{\frac{2n}{n-2}}dv_g]^\frac{n-2}{n}\\
    &\le C_*T_0[\int_{B_2(0)}|\nabla_g (\varphi^{n+1} b^\frac12)|^2dv_g+ \int_{B_2(0)} b(|H|^2+ {T_{g^{-1}}})\varphi^{2n} dv_g],
\end{aligned}
\end{equation}
 where $C_*=C_f\cdot(\int_{B_1(0)}dv_g)^{\frac{2}{n}}$, $T_0=1+(\int_{B_2(0)}T_{g^{-1}}dv_g)^{\frac1{n(n-1)}}$. By Lemma \ref{lemma_00} the mean curvature $H$ is bounded by $T_{g^{-1}}$. Let $\hat\lambda_i=\lambda_i/f$,  $\hat\sigma_k=\sigma_k(\hat\lambda_1,\cdots,\hat\lambda_n)$. By Wang-Yuan \cite{WY11}, Theorem 1.1, step 3,
$$
T_{g^{-1}} Vf^{2-n}=
\cos\Theta\sum_{1\le2k+1\le n}(-1)^k(n-2k)\hat\sigma_{2k}-\sin\Theta\sum_{0\le2k\le n}(-1)^k(n-2k+1)\hat\sigma_{2k-1},
$$
Then by (\ref{eq_14}) (\ref{eq_151}) and the fact that $b>1$,
\begin{equation}\label{eq_meanvalue b}
b(0)\le C_*T_0[\int_{B_2(0)}b(1+\cdots+\sigma_{n-1})\varphi^{2n}dx+C_f\int_{{B}_2(0)}  \varphi^2dv_g].
\end{equation}
Here we need the divergence structure of $\sigma_k$,
$$
k\sigma_k(D^2u)=\sum_{i,j}\frac{\partial}{\partial x_i}(\frac{\partial\sigma_k}{\partial u_{ij}}\frac{\partial u}{\partial x_j})=div({L_{\sigma_k}Du}).
$$
$$
\int_{{B}_2(0)}b\sigma_k\varphi^{2n} dx=-\frac1k\int_{{B}_2(0)}<\varphi^{2n} D b+2nb\varphi^{2n-1}D\varphi,L_{\sigma_k}D u>dx.
$$
Since $|<D b,L_{\sigma_k}Du>|\le C(n)|D u| (|\nabla_g b|^2+1)V$, $\partial\sigma_k/\partial\lambda_i\ge0$ for all $1\le i\le n$ and all $1\le k\le n-1$ by Wang-Yuan \cite{WY11}, Lemma 2.1. Recall the gradient estimates (\ref{eq_14}), we have
$$
\int_{{B}_2(0)}b\sigma_k\varphi^{2n} dx\le C(n)||Du||_{L^\infty(B_2(0))}(\int_{{B}_2(0)}b\sigma_{k-1}\varphi^{2n-1}dx + C_f\int_{{B}_2(0)}dv_g).
$$
Notice that $b\le A+\Delta\cdot$, by iteration and the definition of $C_*, T_0$, we get
$$
b(0)\le C_f\cdot(1+||Du||^n_{L^\infty(B_2(0))})[1+(\int_{{B}_{2}(0)}dv_g)^{1+\frac3n}].
$$
Repeat the above process with $b=1$ to estimate the volume$\int dv_g$,  we get
$$
b(0)\le C(n,||Du||_{L^\infty(B_3(0))},||f||_{C^{0,1} ({ B}_{3}(0))}).
$$
\end{proof}

\subsection{A Prior Interior Gradient Estimates}
\begin{lemma}\label{lemma_23}
Let $u$ satisfies equation (\ref{gslag1}) with $f>0$. If $\lambda_n/f\le-1$, we have
$$
\frac{\lambda_i}{(f^2+\lambda_i^2)}+\frac{\lambda_n}{(f^2+\lambda_n^2)}\le 0,\ 1\le i\le n-1.
$$
\end{lemma}
\begin{proof}
We have
$\frac{\lambda_i}{f^2+\lambda_i^2}= \frac{1}{2f}\sin2\theta_i,$ $1\le i\le n$. Then $\pi+2\theta_n=\sum^{n-1}_{i=1}(\pi-2\theta_i)$, $2\theta_i\ge -2\theta_n$ by Wang-Yuan \cite{WY11}, Lemma 2.1, thus $\sin(2\theta_i)\le-\sin(2\theta_n)$ if $\theta_n\le -\frac\pi4$.
\end{proof}

\subsection{Proof of Theorem \ref{Thm2}}
\begin{proof}
The proof follows Warren-Yuan \cite{WY10}.
Set $M=\mathop{osc}\limits_{B_1(0)}u$. We may assume $0<M\le u\le 2M$ in $B_1(0)$. Let
$$ w=\eta|Du|+Au^2+Bu$$
with $\eta=1-|x|^2$, $A=n/M$, $B=1+n||Df/f||_{L^\infty(B_1(0))}$. If $w$ attains its maximum on $\partial B_1(0)$, the conclusion is straightforward. Otherwise let $x^* \subset B_1(0)$ be a maximum point, choose a coordinate system such that $D^2u(x^*)$ is diagonal. Assume that $u_n(x^*)\ge \frac{|Du(x^*)|}{\sqrt n}>0$, we have at $x^*$
$$0=w_i=\eta|Du|_i+\eta_i|Du|+2Auu_i+ Bu_i,$$
which leads to
$$ |Du|_i=\frac{u_i\lambda_i}{|Du|}=-\frac{\eta_i|Du|+2Auu_i+Bu_i}{\eta}.$$
In particular we get $\lambda_n<0$ at $x^*$. More precisely,
\begin{align}\label{eq_temp_41}
 &\lambda_n \ge -\frac{1}{\eta}\frac{|Du|^2}{|u_n|}(|2Au|+B+|\eta_n|)
               \ge -\sqrt n(B+4n+2)\frac{|Du|}{\eta},\\
&\lambda_n \le -\frac{1}{\eta}\frac{|Du|^2}{|u_n|}(\frac{2Au}{\sqrt n}-|\eta_n|)\label{eq_temp_42}
               \le -(2\sqrt n-2)\frac{|Du|}{\eta}.
\end{align}
If $\lambda_n\ge-f$, we already get the conclusion by (\ref{eq_temp_42}). Thus in the following we assume that $\lambda_n<-f$. According to Wang-Yuan \cite{WY11}, Lemma 2.1, we may assume that $\lambda_1 \ge \cdots \ge \lambda_{n-1} \ge |\lambda_n|$ at $x^*$. We denote
$\overline\Delta_g = g^{ij}\partial_{ij}$, by equality (\ref{eq_uijj}) and Lemma \ref{lemma_23} we have
\begin{equation}
\begin{aligned}\label{eq_temp_44}
\eta\overline\Delta_g|Du|
            &=\sum_{i=1}^n[\eta\frac{(|Du|^2-u_i^2)\lambda_i^2 g^{ii}}{|Du|^3} +\eta\sum_{k=1}^n \frac{u_k \lambda_{i,k}g^{ii}}{|Du|}]\\
            &\ge\sum_{i=1}^n\eta\sum_{k=1}^n \frac{f_k}{f} \frac{u_k}{|Du|}\lambda_i g^{ii}
            \ge -C(n)B|\lambda_n|g^{nn}.
\end{aligned}
\end{equation}
\noindent Finally we can estimate $|Du|$ at $x^*$ using (\ref{eq_temp_41}) and (\ref{eq_temp_44}),
\begin{align*}
0&\ge \Delta_g w
                 \ge |Du|\overline\Delta_g\eta+ \eta\overline\Delta_g|Du| + \sum_{i=1}^n 2 g^{ii}\eta_i |Du|_i
                 + 2Au\overline\Delta_gu+B\overline\Delta_gu+ \sum_{i=1}^n2Ag^{ii}u_i^2\\
                 & \ge -2ng^{nn}|Du| -C(n)B|\lambda_n|g^{nn} -\sum_{i=1}^n2 g^{ii}\eta_i \frac{1}{\eta}(\eta_i|Du|+2Auu_i+Bu_i)
                  +2Ag^{nn}|u_n|^2\\
                 & \ge -2ng^{nn}|Du| -C(n)(B^2+1)\frac{1}{\eta}g^{nn}|Du| -C(n)(B+1)\frac{1}{\eta}g^{nn}|Du|
                  +\frac{2}{M}g^{nn}|Du|^2.
\end{align*}
Thus when $\lambda_n< -f(x^*)$ we have
$$ \frac{2}{M}\eta|Du(x^*)| -  C(n)(B^2+1)\le 0.$$
Recall $|Du(x^*)|\le C(n)f(x^*)$ when $\lambda_n\ge- f(x^*)$, we have
$$|Du(0)|\le C(n)(\|\frac{Df}{f}\|^2_{L^\infty(B_1(0))}+||f||_{L^\infty(B_1(0))}+1)\mathop{osc}\limits_{B_1(0)}u.
$$
Finally we scale by $\hat u=u/||f||_{L^\infty(B_1(0))}$, $\hat f=f/||f||_{L^\infty(B_1(0))}$ and get the result.
\end{proof}

\section {Dirichlet Problem for Generalized Slag}
We denote $F(D^2u)=\sum^n_{k=1}\arctan{\lambda_k}$. In the case of special Lagrangian equation, we have the fact that
$$
F_{m_{ij}}D_{ij}u_{\gamma\gamma}\ge0,
$$
due to the critical and supercritical phase condition. This means that the maximal point of $u_{\gamma\gamma}$ will attain at the boundary, thus we can get the $C^{1,1}$ estimates with the help of boundary data. Another observation made in Wang-Yuan \cite{WY11} is $\Delta_g \lambda_{\max}\ge 0$.

There are problems in getting similar results for the generalized special Lagrangian equation even if $f\in C^{2}$. Notice equality (\ref{eq_33}), we have for some bounded $C_0$
$$
F_{m_{ij}}D_{ij}u_{\gamma\gamma}=\sum^n_{k=1}\frac2{f^2+\lambda_k^2}[\lambda_k\lambda^2_{k,\gamma}
+(-\lambda_k^2f_\gamma+f^2f_\gamma)\lambda_{i,\gamma}]_*+C_0.
$$
we can only expect
$$
F_{m_{ij}}D_{ij}u_{\gamma\gamma}\ge -\delta^2\lambda_1.
$$
In this way we can not restrict $\lambda_{\max}$ to the boundary. We turn to another observation. We have proved the trace Jacobi inequality (\ref{eq_Trace_Jacobin1}), we can get a weaker conclusion,
$$
\Delta_g \max( |A+\Delta u|+K(|x|^2-1)\ge 0,
$$
for $K=K(n,||f||_{C^{1,1}})$ large.
Using the argument in a lecture note on special Lagrangian equations by Yu Yuan, 2016, we get an existence theorem for $f\in C^{2}$, by approximation it leads to $f\in C^{0,1}$.
\begin{theorem}\label{prop_existence}
Let $f(x)\in C^2(\overline B_1(0))$, $f\ge1$. Let $\varphi(x)\in C^4(\partial B_1(0)),\alpha\in (0,1)$. The following Dirichlet problem
$$ \left\{
\begin{array}{lcl}
F(D^2v,x)=\sum^n_{i=1}\arctan\frac{\lambda_i}{f}=\Theta      &      & B_1(0),\\
v=\varphi      &      & \partial B_1(0).
\end{array} \right. $$
has an unique solution $v\in C^{2,\alpha}( \overline B_1(0))$ if the constant $\Theta\in(\frac\pi2(n-2),\frac\pi2n)$.
\end{theorem}
\begin{proof}
The uniqueness comes from the comparison principle. For the existence, let $t\in[0,1]$, $f_t=1-t+tf$. Consider the following Dirichlet problem,
\begin{equation}\label{eq_Dirichlet_t}
 \left\{
\begin{array}{lcl}
F(D^2u,x)=\sum^n_{i=1}\arctan\frac{\lambda_i}{f_t}=\Theta      &      & B_1(0),\\
u=\varphi      &      & \partial B_1(0).
\end{array} \right.
\end{equation}
Let $I=\{t\in[0,1]|\ (\ref{eq_Dirichlet_t})\ \mathrm{has \ a \ solution}\ u_t\in C^{2,\alpha}(\overline B_1(0)) \}$. Notice that $0\in I$. We prove that $I$ is both closed and open. Let $t_0\in I$. It follows from Schauder theory that the linearized operator of $F$ is invertible on $\overline B_1(0)$. By implicit function theorem, (\ref{eq_Dirichlet_t}) is solvable near $t_0$, thus $I$ is open. The fact that $I$ is closed comes from the following  $C^{2,\alpha}$ estimates,
$$
||u||_{C^{2,\alpha}(\overline B_1(0))}\le C(\alpha,\Theta,\mathrm{osc}_{ B_1(0)}u,||f||_{ C^{2}(\overline B_1(0))},||\varphi||_{ C^4(\partial B_1(0))}).
$$
$L^\infty$ bound.\\
$$
-|\varphi|_\infty+\frac12|f|_\infty\tan\frac\Theta{2n}(|x|^2-1)\le u\le |\varphi|_\infty+\frac12\tan\frac\Theta{2n}(|x|^2-1).
$$
Gradient bound.\\

Notice that $\lambda_n\ge -\cot[\Theta-\frac\pi2(n-2)]$, by differentiating equation (\ref{gslag2}) once we have
\begin{equation}\label{eq_gradient_est}
F_{m_{ij}}D_{ij}[u_\gamma+cu+C(|x|^2-1)]\ge 0,\ F_{m_{ij}}D_{ij}[u_\gamma-cu-C(|x|^2-1)]\le0,
\end{equation}
where $c,C$ depend on $\Theta,||f||_{C^{1}}$, the region is implicit. Thus we only need to estimate $|Du|$ at the boundary. We have
\begin{equation}\label{eq_gradient_est1}
L_-\le\varphi \le L_+,
\end{equation}
for some linear functions $L_+, L_-$ depending on $||\varphi||_{C^{2}}$, "$=$'' holds at $y\in\partial B_1(0))$. Let $e\in \partial B_1(0)$, we estimate that at $y$,
$$
    D_eL_-(y)\le u_e(y)\le D_eL_+(y).
$$
We have $|Du|_\infty\le C(\Theta,||f||_{C^{1}},|\varphi||_{C^{2}})$.\\

Hessian bound.\\

First we estimate the double tangential derivatives $u_{TT}$ and the mixed derivatives $u_{TN}$ at the boundary. Notice (\ref{eq_gradient_est}) and the fact that $u_T=\varphi_\theta$ at the boundary, we have
$|u_{TT}|_\infty,|u_{TN}|_\infty\le C(\Theta,||f||_{C^{1}},||\varphi||_{C^3})$ similar to the gradient bound. We have
$$
D^2u\sim\left[
\begin{array}{cc}
 \lambda' & u_{TN}  \\
 u_{TN} & u_{NN} \\
\end{array}
\right].
$$
We assume that $u_{NN}$ is sufficiently large (depending also on $\Theta-\frac\pi2(n-2)$). By the linear algebra, we have $f'(\lambda')=\arctan (\lambda'/f)\ge \frac\pi2(n-3)$. Let $y_0$ be the point such that $f'(\lambda')$ attains its minimal at $\partial B_1(0)$. By Yuan \cite{Yua06}, Lemma 2.1, the level set $\{\lambda'|f'(\lambda')=0\}$ is convex, we have at $y_0$
$$
<Df'(\lambda_0'),\lambda'-\lambda_0'>\ge0,
$$
or
$$
<Df'(\lambda_0'),\lambda'>\ge <Df'(\lambda_0'),\lambda_0'>=c_0=c_0(\varphi_{\theta\theta},u_r).
$$
From the above inequality and the relation $(n-1)u_r+\mathrm{tr}D^2\varphi|_T=\mathrm{tr}D^2u|_T$ we get
$$
u_r\ge \frac{c_0}{n-1}-\frac1{n-1}\mathrm{tr}D^2\varphi|_T, \ ``=" \ at \ y_0,
$$
where $r$ denotes the radius variable. Then by (\ref{eq_gradient_est}) there exists a linear function $L_0$ depending on $\Theta$, $||f||_{C^{1}},||\varphi||_{C^4}$  such that
$$
u_r\ge L_0,
$$
``=" holds at $y_0$. We have for $r<1$
$$
\frac{u_r-u_r(y_0)}{r-1}\le \frac{L_0-L_0(y_0)}{r-1},
$$
and $u_{rr}(y_0)\le  C(\Theta,||f||_{C^{1}},||\varphi||_{C^4})$. Finally we have
\begin{align*}
    \lim_{a\rightarrow \infty}F(D^2u+ae_N\otimes e_N,y_0)\ge &F(D^2u+100e_N\otimes e_N,y_0)\ge F(D^2u,y_0)+\delta_0=\Theta+\delta_0,\\
&\arctan [\lambda'/f](y_0)\ge \Theta-\frac\pi2+ \frac{\delta_0}2,
\end{align*}
where $\delta_0=\delta_0(\Theta,||f||_{C^{1}},||\varphi||_{C^4})>0$. Then for $x\in \partial B_1(0)$, assume that $u_{rr}$ is large,
\begin{align*}
    F (D^2u,x)&=\arctan[(\lambda'+o(1))/f]+\arctan[(u_{nn}+O(1))/f]\\
    &\ge \Theta-\frac\pi2+ \frac{\delta_0}4+o(1)+\arctan[(u_{nn}+O(1))/f].
\end{align*}
Thus $u_{nn}$ is bounded, otherwise there will be a contradiction.
\end{proof}

\begin{remark}
We did not use the fact that $f\in C^{2}(\overline B_1(0))$ when estimating $D^2u$ at the boundary, but we need it to narrow our estimates down to the boundary. Since we have derived the interior $C^{1,1}$ estimates (and $C^{2,\alpha}$ by Caffarelli \cite{Caf89}) for equation (\ref{gslag1}), by approximation we prove Theorem \ref{coro_existence}.
\end{remark}


\section*{Acknowledgements}
The author would like to express great gratitude to professor Yu Yuan, for proposing this subject to us, and further spends considerable time guiding and discussing with us this study through all these five years.

%
%

\appendix
\section {The diagonalize lemma}

Let \{$a_1,\cdots a_n$\} be positive constants, $n\ge1$, $L\in \mathbb{R}^n$. Consider the following $n\times n$ symmetric matrix,
\begin{equation}\label{eq_matrix_1}
\Lambda=\sum^{n}_{i=1}a_ie_i^Te_i-L^TL.
\end{equation}
The fact that $\Lambda\ge0$ is equivalent to the following inequality,
$$
1-\sum^{n}_{i=1}\frac{1}{a_i}|<L,e_i>|^2\ge 0.
$$
We refer to Shankar-Yuan \cite{SY22} and our recent work on Lagrangian mean curvature equations. Here we give a brief indication.

\begin{proof}
Suppose that $a_n\ge a_i>0$, $1\le i\le n-1$. Let $\eta_i=-\sqrt{a_n-a_i}$,
\begin{equation}\label{eq_matrix_1}
\Lambda=a_nI_n-\sum^{n-1}_{i=1}(a_n-a_i)e_i^Te_i-L^TL.
\end{equation}
Consider
$$
\left[
\begin{array}{llll}
 a_1\cdots & \cdots & 0 & \eta_1<L,e_1> \\
 \vdots   & \ & \vdots  &  \vdots \\
 0\cdots    & \cdots & a_{n-1} & \eta_{n-1}<L,e_{n-1}> \\
 \eta_1<L,e_1> & \cdots  & \eta_{n-1}<L,e_{n-1}>  & a_n-|L|^2
\end{array}
\right].
$$
$\Lambda$ has no negative eigenvalue is equivalent to the following inequality,
$$
a_n-|L|^2-\sum^{n-1}_{i=1}\eta_i^2|<L,e_i>|^2a_i^{-1}=a_n-\sum^{n}_{i=1}\frac{a_n}{a_i}|<L,e_i>|^2\ge 0.
$$
\end{proof}


\end{document}